\documentclass[11pt,twoside,english,british]{article}
\usepackage[latin9]{inputenc}
\usepackage[a4paper]{geometry}
\usepackage{fancyhdr}
\usepackage{color}
\usepackage{babel}
\usepackage{verbatim}
\usepackage{refstyle}
\usepackage{mathtools}
\usepackage{enumitem}
\usepackage{amsmath}
\usepackage{amsthm}
\usepackage{amssymb}
\usepackage{setspace}
\usepackage[authoryear]{natbib}
\usepackage{xargs}[2008/03/08]
\usepackage[unicode=true,
 bookmarks=true,bookmarksnumbered=false,bookmarksopen=false,
 breaklinks=true,pdfborder={0 0 0},pdfborderstyle={},backref=false,colorlinks=true]
 {hyperref}
\hypersetup{
 linkcolor=black, citecolor=black, urlcolor=blue}

\makeatletter

\AtBeginDocument{\providecommand\eqref[1]{\ref{eq:#1}}}
\AtBeginDocument{\providecommand\subsecref[1]{\ref{subsec:#1}}}
\AtBeginDocument{\providecommand\secref[1]{\ref{sec:#1}}}
\AtBeginDocument{\providecommand\enuref[1]{\ref{enu:#1}}}
\AtBeginDocument{\providecommand\propref[1]{\ref{prop:#1}}}
\AtBeginDocument{\providecommand\assref[1]{\ref{ass:#1}}}
\AtBeginDocument{\providecommand\appref[1]{\ref{app:#1}}}
\AtBeginDocument{\providecommand\thmref[1]{\ref{thm:#1}}}
\AtBeginDocument{\providecommand\remref[1]{\ref{rem:#1}}}
\AtBeginDocument{\providecommand\lemref[1]{\ref{lem:#1}}}
\RS@ifundefined{subsecref}
  {\newref{subsec}{name = \RSsectxt}}
  {}
\RS@ifundefined{thmref}
  {\def\RSthmtxt{theorem~}\newref{thm}{name = \RSthmtxt}}
  {}
\RS@ifundefined{lemref}
  {\def\RSlemtxt{lemma~}\newref{lem}{name = \RSlemtxt}}
  {}

\numberwithin{equation}{section}
\theoremstyle{remark}
\newtheorem*{notation*}{\protect\notationname}
\theoremstyle{plain}
\newtheorem{assumption}{\protect\assumptionname}
\theoremstyle{remark}
\newtheorem{rem}{\protect\remarkname}[section]
\theoremstyle{plain}
\newtheorem{prop}{\protect\propositionname}[section]
\theoremstyle{plain}
\newtheorem{thm}{\protect\theoremname}[section]
\theoremstyle{plain}
\newtheorem{lem}{\protect\lemmaname}[section]

\setlist[enumerate,1]{label=\upshape{(\roman*)}, ref=(\roman*)}
\setlist[enumerate,2]{label=\upshape{(\alph*)}, ref=(\alph*)}
\setlist[enumerate,3]{label=\upshape{\roman*.}, ref=\roman*}

\makeatletter
  \@ifpackageloaded{refstyle}{}{\usepackage{refstyle}}
\makeatother

\newref{thm}{name = Theorem~}
\newref{prop}{name = Proposition~}
\newref{lem}{name = Lemma~}
\newref{cor}{name = Corollary~}
\newref{ass}{}
\newref{exa}{name = Example~}
\newref{rem}{name = Remark~}
\newref{def}{name = Definition~}

\newref{eq}{refcmd = \text{\upshape(\refeq{#1})}}

\newref{sec}{name = Section~}
\newref{subsec}{name = Section~}
\newref{sub}{name = Section~}
\newref{app}{name = Appendix~}

\newref{fig}{name = Figure~}
\newref{tbl}{name = Table~}

\newref{enu}{name = {}}

\usepackage[user]{zref}

\makeatletter
  \def\olabel#1#2{\begingroup
     \def\@currentlabel{#2}%
     \zlabel{#1}\endgroup
  }
\makeatother

\newcommand{\objlabel}[2]{\olabel{obj:#2}{#1}#1}%
\newcommand{\objref}[1]{\zref{obj:#1}}

\makeatletter
  \@ifpackageloaded{mathrsfs}{}{\usepackage{mathrsfs}}
\makeatother

\date{}

\usepackage{nextpage}

\allowdisplaybreaks[1]

\makeatletter
\g@addto@macro\bfseries{\boldmath}
\makeatother

\numberwithin{equation}{section}

\makeatletter
\newcommand{\lowermath}[1]{\mathpalette{\raisem@th{#1}}}
\newcommand{\raisem@th}[3]{\raisebox{-#1}{$#2#3$}}
\makeatother

\usepackage{thmtools}
\usepackage{thm-restate}

\newcounter{sectemp}
\newcounter{eqtemp}

\newcommand{\restoreposition}[1]{
  \setcounter{section}{\value{sectemp}}
  \setcounter{equation}{\value{eqtemp}}
}

\usepackage{mleftright}
\renewcommand{\left}{\mleft}
\renewcommand{\right}{\mright}

\providecommand{\LyX}{L\kern-.0467em\lower.25em\hbox{Y}\kern-.025emX\@} 

\mathtoolsset{showonlyrefs=true,showmanualtags=true}

\makeatletter
\newcommand*{\getlength}[1]{\strip@pt#1}
\makeatother

\newlength{\tempjot}

\makeatletter
\newcommand \cellfill {\leavevmode \leaders \hb@xt@ .44em{\hss .\hss }\hfill \kern \z@}
\makeatother

\usepackage{scalefnt}
\newcommand\smaller[2][0.85]{{\scalefont{#1}#2}}

\usepackage{alphalph}

\makeatother

\addto\captionsbritish{\renewcommand{\assumptionname}{Assumption}}
\addto\captionsbritish{\renewcommand{\lemmaname}{Lemma}}
\addto\captionsbritish{\renewcommand{\notationname}{Notation}}
\addto\captionsbritish{\renewcommand{\propositionname}{Proposition}}
\addto\captionsbritish{\renewcommand{\remarkname}{Remark}}
\addto\captionsbritish{\renewcommand{\theoremname}{Theorem}}
\addto\captionsenglish{\renewcommand{\assumptionname}{Assumption}}
\addto\captionsenglish{\renewcommand{\lemmaname}{Lemma}}
\addto\captionsenglish{\renewcommand{\notationname}{Notation}}
\addto\captionsenglish{\renewcommand{\propositionname}{Proposition}}
\addto\captionsenglish{\renewcommand{\remarkname}{Remark}}
\addto\captionsenglish{\renewcommand{\theoremname}{Theorem}}
\providecommand{\assumptionname}{Assumption}
\providecommand{\lemmaname}{Lemma}
\providecommand{\notationname}{Notation}
\providecommand{\propositionname}{Proposition}
\providecommand{\remarkname}{Remark}
\providecommand{\theoremname}{Theorem}

\begin{document}
\selectlanguage{english}%


\global\long\def\uwrite#1#2{\underset{#2}{\underbrace{#1}} }%

\global\long\def\blw#1{\ensuremath{\underline{#1}}}%

\global\long\def\abv#1{\ensuremath{\overline{#1}}}%

\global\long\def\vect#1{\mathbf{#1}}%


\global\long\def\smlseq#1{\{#1\} }%

\global\long\def\seq#1{\left\{  #1\right\}  }%

\global\long\def\smlsetof#1#2{\{#1\mid#2\} }%

\global\long\def\setof#1#2{\left\{  #1\mid#2\right\}  }%


\global\long\def\goesto{\ensuremath{\rightarrow}}%

\global\long\def\ngoesto{\ensuremath{\nrightarrow}}%

\global\long\def\uto{\ensuremath{\uparrow}}%

\global\long\def\dto{\ensuremath{\downarrow}}%

\global\long\def\uuto{\ensuremath{\upuparrows}}%

\global\long\def\ddto{\ensuremath{\downdownarrows}}%

\global\long\def\ulrto{\ensuremath{\nearrow}}%

\global\long\def\dlrto{\ensuremath{\searrow}}%


\global\long\def\setmap{\ensuremath{\rightarrow}}%

\global\long\def\elmap{\ensuremath{\mapsto}}%

\global\long\def\compose{\ensuremath{\circ}}%

\global\long\def\cont{C}%

\global\long\def\cadlag{D}%

\global\long\def\Ellp#1{\ensuremath{L^{#1}}}%


\global\long\def\naturals{\ensuremath{\mathbb{N}}}%

\global\long\def\reals{\mathbb{R}}%

\global\long\def\complex{\mathbb{C}}%

\global\long\def\rationals{\mathbb{Q}}%

\global\long\def\integers{\mathbb{Z}}%


\global\long\def\abs#1{\ensuremath{\left|#1\right|}}%

\global\long\def\smlabs#1{\ensuremath{\lvert#1\rvert}}%
 
\global\long\def\bigabs#1{\ensuremath{\bigl|#1\bigr|}}%
 
\global\long\def\Bigabs#1{\ensuremath{\Bigl|#1\Bigr|}}%
 
\global\long\def\biggabs#1{\ensuremath{\biggl|#1\biggr|}}%

\global\long\def\norm#1{\ensuremath{\left\Vert #1\right\Vert }}%

\global\long\def\smlnorm#1{\ensuremath{\lVert#1\rVert}}%
 
\global\long\def\bignorm#1{\ensuremath{\bigl\|#1\bigr\|}}%
 
\global\long\def\Bignorm#1{\ensuremath{\Bigl\|#1\Bigr\|}}%
 
\global\long\def\biggnorm#1{\ensuremath{\biggl\|#1\biggr\|}}%


\global\long\def\Union{\ensuremath{\bigcup}}%

\global\long\def\Intsect{\ensuremath{\bigcap}}%

\global\long\def\union{\ensuremath{\cup}}%

\global\long\def\intsect{\ensuremath{\cap}}%

\global\long\def\pset{\ensuremath{\mathcal{P}}}%

\global\long\def\clsr#1{\ensuremath{\overline{#1}}}%

\global\long\def\symd{\ensuremath{\Delta}}%

\global\long\def\intr{\operatorname{int}}%

\global\long\def\cprod{\otimes}%

\global\long\def\Cprod{\bigotimes}%


\global\long\def\smlinprd#1#2{\ensuremath{\langle#1,#2\rangle}}%

\global\long\def\inprd#1#2{\ensuremath{\left\langle #1,#2\right\rangle }}%

\global\long\def\orthog{\ensuremath{\perp}}%

\global\long\def\dirsum{\ensuremath{\oplus}}%


\global\long\def\spn{\operatorname{sp}}%

\global\long\def\rank{\operatorname{rk}}%

\global\long\def\proj{\operatorname{proj}}%

\global\long\def\tr{\operatorname{tr}}%

\global\long\def\diag{\operatorname{diag}}%


\global\long\def\smpl{\ensuremath{\Omega}}%

\global\long\def\elsmp{\ensuremath{\omega}}%

\global\long\def\sigf#1{\mathcal{#1}}%

\global\long\def\sigfield{\ensuremath{\mathcal{F}}}%
\global\long\def\sigfieldg{\ensuremath{\mathcal{G}}}%

\global\long\def\flt#1{\mathcal{#1}}%

\global\long\def\filt{\mathcal{F}}%
\global\long\def\filtg{\mathcal{G}}%

\global\long\def\Borel{\ensuremath{\mathcal{B}}}%

\global\long\def\cyl{\ensuremath{\mathcal{C}}}%

\global\long\def\nulls{\ensuremath{\mathcal{N}}}%

\global\long\def\gauss{\mathfrak{g}}%

\global\long\def\leb{\mathfrak{m}}%


\global\long\def\prob{P}%

\global\long\def\Prob{\ensuremath{\mathbb{P}}}%

\global\long\def\Probs{\mathcal{P}}%

\global\long\def\PROBS{\mathcal{M}}%

\global\long\def\expect{\mathbb{E}}%

\global\long\def\probspc{\ensuremath{(\smpl,\filt,\Prob)}}%


\global\long\def\iid{\ensuremath{\textnormal{i.i.d.}}}%

\global\long\def\as{\ensuremath{\textnormal{a.s.}}}%

\global\long\def\asp{\ensuremath{\textnormal{a.s.p.}}}%

\global\long\def\io{\ensuremath{\ensuremath{\textnormal{i.o.}}}}%

\newcommand\independent{\protect\mathpalette{\protect\independenT}{\perp}}
\def\independenT#1#2{\mathrel{\rlap{$#1#2$}\mkern2mu{#1#2}}}

\global\long\def\indep{\independent}%

\global\long\def\distrib{\ensuremath{\sim}}%

\global\long\def\distiid{\ensuremath{\sim_{\iid}}}%

\global\long\def\asydist{\ensuremath{\overset{a}{\distrib}}}%

\global\long\def\inprob{\ensuremath{\overset{p}{\goesto}}}%

\global\long\def\inprobu#1{\ensuremath{\overset{#1}{\goesto}}}%

\global\long\def\inas{\ensuremath{\overset{\as}{\goesto}}}%

\global\long\def\eqas{=_{\as}}%

\global\long\def\inLp#1{\ensuremath{\overset{\Ellp{#1}}{\goesto}}}%

\global\long\def\indist{\ensuremath{\overset{d}{\goesto}}}%

\global\long\def\eqdist{=_{d}}%

\global\long\def\wkc{\ensuremath{\rightsquigarrow}}%

\global\long\def\fdd{\ensuremath{\rightsquigarrow_{\mathrm{fdd}}}}%

\global\long\def\wkcu#1{\overset{#1}{\ensuremath{\rightsquigarrow}}}%

\global\long\def\plim{\operatorname*{plim}}%


\global\long\def\var{\operatorname{var}}%

\global\long\def\lrvar{\operatorname{lrvar}}%

\global\long\def\cov{\operatorname{cov}}%

\global\long\def\corr{\operatorname{corr}}%

\global\long\def\bias{\operatorname{bias}}%

\global\long\def\MSE{\operatorname{MSE}}%

\global\long\def\med{\operatorname{med}}%


\global\long\def\simple{\mathcal{R}}%

\global\long\def\sring{\mathcal{A}}%

\global\long\def\sproc{\mathcal{H}}%

\global\long\def\Wiener{\ensuremath{\mathbb{W}}}%

\global\long\def\sint{\bullet}%

\global\long\def\cv#1{\left\langle #1\right\rangle }%

\global\long\def\smlcv#1{\langle#1\rangle}%

\global\long\def\qv#1{\left[#1\right]}%

\global\long\def\smlqv#1{[#1]}%


\global\long\def\trans{\ensuremath{\prime}}%

\global\long\def\indic{\ensuremath{\mathbf{1}}}%

\global\long\def\Lagr{\mathcal{L}}%

\global\long\def\grad{\nabla}%

\global\long\def\pmin{\ensuremath{\wedge}}%
\global\long\def\Pmin{\ensuremath{\bigwedge}}%

\global\long\def\pmax{\ensuremath{\vee}}%
\global\long\def\Pmax{\ensuremath{\bigvee}}%

\global\long\def\sgn{\operatorname{sgn}}%

\global\long\def\argmin{\operatorname*{argmin}}%

\global\long\def\argmax{\operatorname*{argmax}}%

\global\long\def\Rp{\operatorname{Re}}%

\global\long\def\Ip{\operatorname{Im}}%

\global\long\def\deriv{\ensuremath{\mathrm{d}}}%

\global\long\def\diffnspc{\ensuremath{\deriv}}%

\global\long\def\diff{\ensuremath{\,\deriv}}%

\global\long\def\i{\ensuremath{\mathrm{i}}}%

\global\long\def\e{\mathrm{e}}%

\global\long\def\sep{,\ }%

\global\long\def\defeq{\coloneqq}%

\global\long\def\eqdef{\eqqcolon}%

\selectlanguage{british}%

\selectlanguage{english}%

\global\long\def\S{\mathcal{S}}%

\global\long\def\M{\mathcal{M}}%

\global\long\def\N{\mathcal{N}}%

\global\long\def\V{\mathcal{V}}%

\global\long\def\U{\mathcal{U}}%

\global\long\def\R{\mathcal{R}}%

\global\long\def\mg{\xi}%

\global\long\def\minidx{k_{0}}%

\global\long\def\minidxm{m_{0}}%

\global\long\def\smltail#1{\smlnorm{#1}_{\omega}}%

\global\long\def\alphnorm#1#2{\norm{#1}_{[#2]}}%

\global\long\def\smlalph#1#2{\smlnorm{#1}_{[#2]}}%

\global\long\def\smlmom#1#2{\smlnorm{#1}_{#2}}%

\newcommandx\RV[1][usedefault, addprefix=\global, 1=]{\mathrm{RV}(#1)}%

\global\long\def\BI{\mathrm{BI}}%

\global\long\def\BInorm#1{\smlnorm{#1}_{\BI}}%

\global\long\def\BIalph#1{\mathrm{BI}_{[#1]}}%

\global\long\def\BILalph#1{\mathrm{BIL}_{[#1]}}%

\global\long\def\BImom#1{\mathrm{BI}_{#1}}%

\global\long\def\BILmom#1{\mathrm{BIL}_{#1}}%

\global\long\def\BIz{\mathrm{BI_{0}}}%

\global\long\def\BIc{\mathrm{BIC}}%

\global\long\def\BIcz{\mathrm{BIC}_{0}}%

\global\long\def\BIl{\mathrm{BIL}}%

\global\long\def\BIlz{\mathrm{BIL}_{0}}%

\global\long\def\Lip{\mathrm{Lip}}%

\global\long\def\Reg{\mathrm{BI_{R}}}%

\global\long\def\locest{\mathcal{L}_{n}}%

\global\long\def\loctime{\mathcal{L}}%

\global\long\def\lmest{\mathcal{\mu}_{n}}%

\global\long\def\lmeas{\mathcal{\mu}}%

\global\long\def\minpr{\mathfrak{m}}%

\global\long\def\maxpr{\mathfrak{M}}%

\global\long\def\boundpr{\mathfrak{U}}%

\global\long\def\major#1{\mathcal{#1}}%

\global\long\def\grid{\mathbb{G}}%

\global\long\def\ucc{\mathrm{ucc}}%

\global\long\def\Hspc{\mathscr{H}}%

\global\long\def\Bspc{\mathscr{B}}%

\global\long\def\Aspc{\mathscr{A}}%

\global\long\def\Cspc{\mathscr{C}}%

\global\long\def\gauss{\phi_{\mathfrak{g}}}%

\global\long\def\Gauss{\Phi_{\mathfrak{g}}}%

\global\long\def\cmpct{\varphi}%

\global\long\def\Fset{\mathscr{F}}%

\global\long\def\Gset{\mathscr{G}}%

\global\long\def\Uset{\mathscr{U}}%

\global\long\def\locp{\kappa}%

\global\long\def\signal{S}%

\global\long\def\cfidx{p_{0}}%

\global\long\def\smlfloor#1{\lfloor#1\rfloor}%

\global\long\def\smlceil#1{\lceil#1\rceil}%

\global\long\def\bkt{[\,]}%

\global\long\def\bktpower{\beta}%

\global\long\def\ctrproc{\nu_{n}}%

\newcommand{\SM}{\textnormal{(SM)}}

\newcommand{\AP}{\textnormal{(AP)}}

\newcommand{\LM}{\textnormal{(LM)}}

\newcommand{\IN}{\textnormal{(IN)}}

\newcommand{\CR}{\textnormal{(CR)}}

\global\long\def\diam{\operatorname{diam}}%

\global\long\def\nseq{e}%

\global\long\def\etamom{q_{0}}%

\global\long\def\bandmax{r_{0}}%

\global\long\def\isbigop{\lesssim_{p}}%

\global\long\def\derivbnd#1{\abv m_{#1}}%

\global\long\def\truncnorm#1{\smlnorm{#1}_{n}}%

\global\long\def\intfnal{\iota}%
\selectlanguage{british}%

\global\long\def\genloc{\mu}%

\global\long\def\gauss{\varphi}%

\global\long\def\std{\kappa}%

\global\long\def\Rho{\mathrm{P}}%

\global\long\def\locsmpl{\lambda}%

\global\long\def\avgfn{\genloc}%

\global\long\def\limdens{\nu}%

\global\long\def\rseqs{\mathcal{P}}%

\global\long\def\stdseq{d}%

\global\long\def\scseq{e}%

\global\long\def\BL{\mathrm{BL}}%

\global\long\def\err{\varepsilon}%

\global\long\def\mean{\theta}%

\global\long\def\leqc#1{\leq_{\lowermath{2.5pt}{#1}}}%

\global\long\def\trikern{\kappa}%

\global\long\def\nullhyp{H_{0}}%

\global\long\def\althyp{H_{1}}%

\global\long\def\nullval{\theta}%

\global\long\def\AsySz{\mathrm{AsySz}}%

\global\long\def\AsyMaxCP{\mathrm{AsyMaxCP}}%

\global\long\def\CP{\mathrm{CP}}%

\global\long\def\NRP{\mathrm{NRP}}%

\global\long\def\ci{\mathcal{C}}%

\global\long\def\regfn{m}%

\global\long\def\Mset{\mathscr{M}}%

\global\long\def\asymprob{\asymp_{p}}%

\global\long\def\hrate{\mathfrak{h}}%

\global\long\def\Xset{\mathscr{X}}%

\global\long\def\estdom{A_{n}}%

\global\long\def\Hset{\mathscr{H}}%

\newcommand{\KAP}{\citetalias{KAP15JoE}}

\defcitealias{KAP15JoE}{KAP}

\newcommand{\WP}{\citetalias{WP09ET}}

\defcitealias{WP09ET}{WP}

\title{Asymptotic Theory for Kernel Estimators under Moderate Deviations
from a Unit Root, with an Application to the Asymptotic Size of Nonparametric
Tests}
\author{James A.\ Duffy\thanks{Corpus Christi College and Department of Economics, University of
Oxford; email: \texttt{james.duffy@economics.ox.ac.uk}. I thank V.~Berenguer-Rico,
S. Mavroeidis, B.~Nielsen, and participants at seminars at Cambridge,
Princeton, Vienna and UCL for comments on an earlier version of this
paper. The manuscript was prepared with \LyX{}~2.3.1 and JabRef~3.8.2.}}
\date{August 2019}
\maketitle
\begin{abstract}
\noindent We provide new asymptotic theory for kernel density estimators,
when these are applied to autoregressive processes exhibiting moderate
deviations from a unit root. This fills a gap in the existing literature,
which has to date considered only nearly integrated and stationary
autoregressive processes. These results have applications to nonparametric
predictive regression models. In particular, we show that the null
rejection probability of a nonparametric $t$ test is controlled uniformly
in the degree of persistence of the regressor. This provides a rigorous
justification for the validity of the usual nonparametric inferential
procedures, even in cases where regressors may be highly persistent.

\vspace{0.8cm}

\noindent \emph{JEL codes}: C14; C22

\vspace{0.4cm}

\noindent \emph{Keywords}: nonparametric regression; predictive regression;
density estimation; uniformly valid inference; moderate deviations
from a unit root; mild integration.
\end{abstract}
\thispagestyle{plain}

\newpage{}

\section{Introduction\label{sec:intro}}

Consider the predictive regression model
\begin{equation}
y_{t}=m(x_{t-1})+u_{t}\label{eq:predintro}
\end{equation}
where $m$ is an unknown function and $u_{t}$ is a martingale difference
sequence. $x_{t}$ is a time series with an unknown \textendash{}
but possibly very high \textendash{} degree of persistence, which
we shall parametrise as
\begin{equation}
x_{t}=\rho x_{t-1}+v_{t},\label{eq:regmodel}
\end{equation}
for $\rho\in\Rho\defeq[-1+\delta,1]$, where $v_{t}$ is weakly dependent.

In this setting, parametric estimators of $m$ are known to have a
limiting distribution that is non-Gaussian, and which depends on the
proximity of $\rho$ to unity. The difficulties that this poses for
inference has spawned a large literature (see e.g.\ \citealp*{CES95ET};
\citealp{CY06JFE}; \citealp{JM06Ecta}; \citealp{MP09}; \citealp{PL13JoE};
and \citealp*{EMW15Ecta}). In contrast, nonparametric estimators
of $m$ have been shown to be asymptotically normal even when regressors
are nearly integrated (see, in particular, \citealp{WP09ET,WP09Ecta}).
Because this is also true when $x_{t}$ is stationary, it has been
recently argued by \citet*[hereafter KAP]{KAP15JoE} that valid inferences
on $m$ may be drawn simply by referring a nonparametric $t$ statistic
to normal critical values. That is, so far as \emph{nonparametric}
inferences are concerned, it is not necessary to make any adjustments
when $\rho$ is close to unity. \KAP{} provide some simulation evidence
in support of this claim.

The primary motivation for the present work is to provide a rigorous
proof of the asymptotic validity of the nonparametric $t$ test, in
the setting of the model \eqref{predintro}\textendash \eqref{regmodel},
thereby putting \KAP{}'s thesis on a surer footing. What do we mean
by `asymptotic validity' in this context? For a test of
\[
\nullhyp:\regfn(x)=\nullval\qquad\text{against}\qquad\althyp:\regfn(x)\neq\nullval
\]
(for a chosen $x\in\reals$), the null rejection probability of the
$t$ test needs to be controlled uniformly over all parameters left
unrestricted by $\nullhyp$: in particular, over all $\rho\in\Rho$.
(For a discussion of this issue in more general contexts, see for
example \citealp{Rom04Sc}, \citealp{Mik07Ecta}, and \citealp*{ACG2011CFDP}.)
More precisely, the nonparametric $t$ test is said to be asymptotically
of size $\alpha$ if
\begin{equation}
\limsup_{n\goesto\infty}\sup_{\rho\in\Rho}\Prob_{\rho}\{\smlabs{\hat{t}_{n}(x)}\geq z_{1-\alpha/2}\}\leq\alpha,\label{eq:asysizecontrol}
\end{equation}
where the `$\rho$' subscript on $\Prob_{\rho}$ indicates the dependence
of this probability on the value of $\rho$ in \eqref{regmodel},
$z_{\tau}$ denotes $\tau$th quantile of the standard normal distribution,
and
\[
\hat{t}_{n}(x)=s_{n}(x)^{-1}[\hat{\regfn}_{n}(x)-\nullval]
\]
denotes the nonparametric $t$ statistic for $\hat{\regfn}_{n}(x)$,
the local level (Nadaraya\textendash Watson) estimator of $m$ at
$x$, and $s_{n}^{2}(x)$ an estimate of its asymptotic variance (see
\subsecref{NPest} below for precise definitions).

Existing limit theory for nonparametric regression estimators establishes
that $\hat{t}_{n}(x)$ is asymptotically normal when $\rho$ is either
fixed and in the stationary region $(\rho<1$), or is local to unity
in the sense that $\rho=1+c/n$ (see, e.g.,\ \citealp{WM02AS,WP09ET,WP09Ecta};
and \KAP{}). As we shall argue in \secref{unifinf}, these results
are sufficient only to establish what might be termed the `pointwise
asymptotic validity' of the $t$ test, i.e.\ that
\[
\limsup_{n\goesto\infty}\Prob_{\rho}\{\smlabs{\hat{t}_{n}(x)}\geq z_{1-\alpha/2}\}\leq\alpha,\ \forall\rho\in\Rho.
\]
To prove \eqref{asysizecontrol}, we additionally need to show that
$\hat{t}_{n}(x)\wkc N[0,1]$ when $x_{t}$ exhibits `moderate deviations
from a unit root', in the sense that $\rho_{n}\goesto1$ but $n(1-\rho_{n})\goesto\infty$;
we refer to these as \emph{mildly integrated} \emph{processes }\citep[see][]{GP06JTSA,PM07JoE}.

Accordingly, \secref{spatialdensity} of this paper provides new asymptotic
theory for sums of integrable transformations of mildly integrated
processes \textendash{} i.e.\ for kernel density estimators applied
to such processes. This fills a significant gap in the existing technical
literature, and allows for a successful proof of \eqref{asysizecontrol}.
The development of this theory relies on an interesting combination
of arguments appropriate to stationary and local-to-unity processes.
The dependence of mildly integrated processes is sufficiently weak
that kernel density estimators converge not to the local time of some
limiting process, but to the standard normal density. In particular,
we have
\begin{equation}
\frac{d_{n}}{nh_{n}}\sum_{t=1}^{n}f\left(\frac{x_{t}-d_{n}a}{h_{n}}\right)\inprob\varphi(a)\int_{\reals}f,\label{eq:milim}
\end{equation}
where $f$ is an integrable function, $d_{n}\defeq\var(x_{n})$, $\varphi(a)\defeq(2\pi)^{-1/2}\e^{-a^{2}/2}$,
and $h_{n}=o(1)$ is a bandwidth sequence. In this respect, mildly
integrated processes are more akin to stationary processes, except
for the noted normality of the limiting density. On the other hand,
they also share the diminished recurrence and slower rates of convergence
characteristic of local-to-unity processes.

The rest of this paper is organised as follows. We begin by outlining
a simplified version of the inferential problem studied by \KAP{}
(Sections~\ref{subsec:dgp}\textendash \ref{subsec:NPest}). We then
provide an explanation of how the asymptotic validity of the $t$
test \textendash{} in the sense of \eqref{asysizecontrol} above \textendash{}
may be established with the aid of new results on integrable transformations
of mildly integrated processes (\subsecref{unif}). These results
are developed in \secref{spatialdensity}. Proofs of the main results
appear in Appendices\ \ref{app:unifproof}\textendash \ref{app:mild}.
Proofs of technical results that are either conceptually straightforward,
or closely related to those that have already appeared in the literature,
are given in the Online Supplement to this article, available at Cambridge
Journals Online (journals.cambridge.org/ect).
\begin{notation*}
All limits are taken as $n\goesto\infty$ unless otherwise stated.
$\reals$ denotes the real numbers. For sequences $\{a_{n}\}$, $\{b_{n}\}$:
$a_{n}\asymp b_{n}$ denotes $\lim_{n\goesto\infty}a_{n}/b_{n}=c\in\reals\backslash\{0\}$,
and $a_{n}\sim b_{n}$ denotes $\lim_{n\goesto\infty}a_{n}/b_{n}=1$.
For positive sequences: $a_{n}\lesssim b_{n}$ denotes $\limsup_{n\goesto\infty}a_{n}/b_{n}<\infty$
\textendash{} equivalently, $a_{n}=O(b_{n})$. For random sequences
$\{x_{n}\}$, $\{y_{n}\}$: $x_{n}\isbigop y_{n}$ denotes $x_{n}=O_{p}(y_{n})$.
$\wkc$ denotes weak convergence in the sense of \citet{VVW96}, and
$\fdd$ the convergence of finite-dimensional distributions. For $x\geq0$,
$\smlfloor x$ denotes the greatest integer less than or equal to
$x$.
\end{notation*}

\section{Nonparametric predictive regression\label{sec:unifinf}}

\subsection{Data generating process\label{subsec:dgp}}

As outlined above, the data generating process (DGP) is the same as
that studied by \KAP{}. We have the following nonlinear predictive
regression model
\begin{equation}
y_{t}=\regfn(x_{t-1})+u_{t}.\label{eq:regression}
\end{equation}
where $\regfn$ and the series $\{x_{t},u_{t}\}$ are assumed to satisfy
the following

\setcounter{assumption}{2901}
\begin{assumption}
\label{ass:reg}~

\begin{enumerate}[{label=\textnormal{\ref{ass:reg}\smaller[0.76]{\arabic*}},leftmargin=1.5cm}]
\item \label{enu:regfn} $\regfn$ is Lipschitz continuous.
\item \label{enu:iidseq} $\{\err_{t}\}$ is a scalar i.i.d.\ sequence;
$\err_{0}$ has a characteristic function $\psi_{\err}(\lambda)\defeq\expect\e^{\i\lambda\err_{0}}$
that is integrable, and a probability density $f_{\err}$ that is
Lipschitz continuous and everywhere nonzero; $\expect\err_{0}=0$
and $\expect\err_{0}^{2}=1$.
\item \label{enu:regproc}$\{x_{t}\}$ and $\{v_{t}\}$ are generated according
to 
\begin{align}
x_{t} & =\rho x_{t-1}+v_{t} & v_{t} & \defeq\sum_{k=0}^{\infty}\phi_{k}\err_{t-k},\label{eq:regproc}
\end{align}
with $x_{0}=0$; $\rho\in\Rho\defeq[-1+\delta,1]$ for some $\delta>0$;
$\phi_{0}\neq0$; $\sum_{k=0}^{\infty}\smlabs{\phi_{k}}<\infty$;
and $\phi\defeq\sum_{k=0}^{\infty}\phi_{k}\neq0$.
\item \label{enu:regdist} $\{u_{t}\}$ is a martingale difference sequence
with respect to $\filtg_{t}\defeq\sigma(\{x_{s},u_{s}\}_{s\leq t})$,
with $\expect[u_{t}^{2}\mid\filtg_{t-1}]=\sigma_{u}^{2}$ a.s.\ constant,
and $\sup_{t}\expect[\smlabs{u_{t}}^{4}\mid\filtg_{t-1}]<\infty$
a.s.
\end{enumerate}
\end{assumption}

\begin{rem}
(a) Our assumptions closely correspond to those of \KAP{}. In particular,
\enuref{regproc} is cognate with their Assumptions~2.3 and 2.4,
with the key difference that we do not restrict $\{x_{t}\}$ to the
local-to-unity region, in which $\rho=1+\tfrac{c}{n}$ for some fixed
$c\in\reals$. We instead allow $\rho$ to range over the entirety
of $\Rho=[-1+\delta,1]$. On the other hand, $\sum_{k=0}^{\infty}\smlabs{\phi_{k}}<\infty$
implies that $\{v_{t}\}$ is a short-memory process, and so excludes
the long-memory and anti-persistent cases that are also considered
in \KAP{}. While it is likely that our results could be extended
to cover these cases, we have excluded these to keep this paper to
a manageable length.

(b) Owing to the initialisation $x_{0}=0$, the regressor process
is nonstationary, regardless of the value of $\rho$. However, \eqref{regproc}
has a stationary solution when $\rho<1$, which corresponds to the
weak limit of $x_{n}$ as $n\goesto\infty$. The assumption of a fixed
initialisation is made only for convenience; our results below would
still hold provided $x_{0}$ is stochastically bounded (and adapted
to $\filtg_{0}$).

(c) The assumption that $f_{\err}$ is Lipschitz is used only in the
stationary region, i.e.\ when $\rho<1$, to facilitate the direct
application of results from \citet*{WHH10SP}. Strict positivity of
$f_{\err}$ is also assumed merely for convenience, to ensure that
the stationary solution to \eqref{regproc} has a density that is
strictly positive at every $x\in\reals$, thereby avoiding any inadvertent
attempts to estimate $\regfn(x)$ at points of zero density. (Aside
from ensuring such points are avoided, this assumption is \emph{not}
needed for \propref{size} below.)
\end{rem}

\subsection{Estimation and inference\label{subsec:NPest}}

\KAP{} develop two nonparametric tests for the `predictability'
of $y_{t}$ by $x_{t-1}$, each of which involve taking either the
average or the maximum of a finite collection of nonparametric $t$
statistics, evaluated at selected points in the domain of the regressor.
Critical values for these tests are derived from the normal distribution,
which is justified if each of the $t$ statistics are asymptotically
normal (and asymptotically independent). In what follows, we consider
a simplified version of their testing problem, which involves testing
hypotheses about the value of $m$ at a single $x\in\reals$, by comparing
a $t$ statistic to normal critical values. The asymptotic validity
of this simplified procedure is of interest in its own right, and
has direct implications for the validity of the predictability tests
developed by \KAP{}.\footnote{These implications are fully developed in an earlier version of this
paper (arXiv:1509.05017v3).}

Following \KAP{}, an estimate of the regression function $\regfn$,
at a chosen $x\in\reals$, is provided by the local level (Nadaraya-Watson)
regression estimator,
\begin{equation}
\hat{\regfn}_{n}(x;h)\defeq\frac{\sum_{t=1}^{n}K_{h}(x_{t}-x)y_{t+1}}{\sum_{t=1}^{n}K_{h}(x_{t}-x)},\label{eq:reghat}
\end{equation}
where $K:\reals\setmap\reals$ is a smooth probability density, $h>0$
denotes the bandwidth, and $K_{h}(x)\defeq h^{-1}K(h^{-1}u)$. For
the purposes of developing the asymptotics of $\hat{m}_{n}$, we shall
suppose $h=h_{n}$, for $\{h_{n}\}$ a bandwidth sequence satisfying\setcounter{assumption}{506}
\begin{assumption}[smoothing]
\label{ass:smoothing}~
\begin{enumerate}[{label=\textnormal{\ref{ass:smoothing}\smaller[0.76]{\arabic*}},leftmargin=1.5cm}]
\item $K$ is non-negative, bounded and Lipschitz, with $\int_{\reals}\smlabs xK(x)\deriv x<\infty$
and $\int_{\reals}K=1$;
\item \emph{\label{enu:smoothing:hn}} $h_{n}=o(1)$ and $n^{1/2}h_{n}\goesto\infty$.
\end{enumerate}
\end{assumption}

\begin{rem}
The maximum rate at which $h_{n}$ may shrink to zero, while still
ensuring the consistency of $\hat{\regfn}_{n}$, will be determined
by the values of $\rho$ for which $\{x_{t}\}$ is least recurrent
\textendash{} i.e.\ when $\rho=1$. This accounts for the requirement
that $n^{1/2}h_{n}\goesto\infty$ in \enuref{smoothing:hn}. This
could be relaxed if $h_{n}$ were chosen so as to adapt to the (unknown)
recurrence of $\{x_{t}\}$.
\end{rem}

For each $x\in\reals$, a test of 
\begin{equation}
\nullhyp:\regfn(x)=\nullval\qquad\text{against}\qquad\althyp:\regfn(x)\neq\nullval\label{eq:nullandalt}
\end{equation}
may then be based on the nonparametric $t$-statistic
\begin{equation}
\hat{t}_{n}(x)\defeq s_{n}(x)^{-1}[\hat{\regfn}(x;h_{n})-\nullval],\label{eq:tstat}
\end{equation}
where 
\begin{align}
s_{n}^{2}(x) & \defeq\frac{\hat{\sigma}_{u}^{2}(x)\int_{\reals}K^{2}}{h_{n}\sum_{t=1}^{n}K_{h_{n}}(x_{t}-x)} & \hat{\sigma}_{u}^{2}(x) & \defeq\frac{\sum_{t=1}^{n}K_{h_{n}}(x_{t}-x)[y_{t+1}-\hat{\regfn}_{n}(x)]^{2}}{\sum_{t=1}^{n}K_{h_{n}}(x_{t}-x)}.\label{eq:s-sig}
\end{align}
As in \KAP{}, critical values for the test are provided by the quantiles
of a standard normal distribution: so that for a test having nominal
size $\alpha$, $\nullhyp$ would be rejected if $\smlabs{\hat{t}_{n}(x)}>z_{1-\alpha/2}$,
where $z_{\tau}$ denotes the $\tau$th quantile of the standard normal
distribution.

\subsection{Asymptotic validity of the $t$ test\label{subsec:unif}}

The purpose of this section is to show that the testing procedure
described above has the correct size asymptotically, in the sense
that the nominal and actual size of the test approximately agree in
large samples.

To that end, recall that the size of a test of is commonly defined
as its maximum rejection probability over all values of the model
parameters consistent with the null hypothesis (see e.g.\ \citealp{LR05},
p.\ 57). In the present setting, $\nullhyp$ restricts only the value
of $\regfn$ (at $x$), leaving the nuisance parameter $\rho\in\Rho$
entirely unrestricted.\footnote{We might also regard other aspects of the model, such as the distributions
of $\epsilon_{t}$ and $u_{t}$, as (infinite-dimensional) nuisance
parameters. The size of the $t$ test would then be more properly
computed by taking the maximum rejection probability over the parameter
space for these distributions (as well as over $\rho\in\Rho$). Our
results could be extended in this direction, but we have refrained
from doing so here in order to keep the paper to a reasonable length.} Thus the $t$ test for $\nullhyp$ has size $\alpha$ asymptotically
if
\begin{equation}
\limsup_{n\goesto\infty}\sup_{\rho\in\Rho}\Prob_{\rho}\{\smlabs{\hat{t}_{n}(x)}\geq z_{1-\alpha/2}\}=\alpha.\label{eq:sizealpha}
\end{equation}
where the `$\rho$' subscript on $\Prob_{\rho}$ makes explicit
the dependence of this probability of the value of $\rho$ in \eqref{regproc}.
It is known from previous work \textendash{} e.g.\ from Lemma~2
in \KAP{} \textendash{} that 
\begin{equation}
\hat{t}_{n}(x)\wkc N[0,1]\label{eq:tcvg}
\end{equation}
for every \emph{fixed} $\rho\in\Rho$, and indeed when $\rho=1+c/n$.
But while this result is highly suggestive, it is insufficient to
establish \eqref{sizealpha}.

What would be sufficient for \eqref{sizealpha}? Since there must
be a sequence $\{\rho_{n}^{\ast}\}\subset\Rho$ such that
\[
\limsup_{n\goesto\infty}\sup_{\rho\in\Rho}\Prob_{\rho}\{\smlabs{\hat{t}_{n}(x)}\geq z_{1-\alpha/2}\}=\lim_{n\goesto\infty}\Prob_{\rho_{n}^{\ast}}\{\smlabs{\hat{t}_{n}(x)}\geq z_{1-\alpha/2}\},
\]
\eqref{sizealpha} will follow once we have shown that \eqref{tcvg}
holds for the drifting sequence $\rho=\rho_{n}^{\ast}$. Rather than
try to characterise $\{\rho_{n}^{\ast}\}$ and show that \eqref{tcvg}
holds for that specific sequence, \propref{size} below establishes
that \eqref{tcvg} holds for \emph{every} drifting sequence $\{\rho_{n}\}\subset\Rho$.
This immediately implies \eqref{sizealpha}, and carries the further
implication that the $t$-test is asymptotically similar, in the sense
that
\[
\liminf_{n\goesto\infty}\inf_{\rho\in\Rho}\Prob_{m,\rho}\{\smlabs{\hat{t}_{n}(x)}\geq z_{1-\alpha/2}\}=\alpha
\]
holds additionally. (For a further discussion, see \citealp{ACG2011CFDP}.)

Our main result on the asymptotic size of the $t$ test may now be
stated. We shall additionally assume $h_{n}=o(n^{-1/3})$, so as to
ensure that the bias in $\hat{\regfn}_{n}$ is asymptotically negligible.\footnote{If \enuref{regfn} were strengthened such that the \emph{second} derivatives
of $m$ were uniformly bounded, then it would be possible to relax
this requirement to $h_{n}=o(n^{-1/6})$: see e.g.\ \citeauthor{WP09Ecta}
(\citeyear{WP09Ecta}, Rem.\ C; \citeyear{WP11ET}).}
\begin{prop}
\label{prop:size}Suppose \assref{reg} and \assref{smoothing} hold,
and that additionally $h_{n}=o(n^{-1/3})$. Then under $\nullhyp$
\begin{equation}
\hat{t}_{n}(x)\wkc N[0,1]\label{eq:tstatcvg}
\end{equation}
along every $\{\rho_{n}\}\subset\Rho$, and the nonparametric $t$
test of \eqref{nullandalt} is asymptotically similar.
\end{prop}
The proof of \propref{size} appears in \appref{unifproof}. The problem
reduces to one of proving that
\begin{equation}
\upsilon_{n}(x)\defeq\frac{h_{n}^{1/2}\sum_{t=1}^{n}K_{h_{n}}(x_{t}-x)u_{t+1}}{\sigma_{u}\left[\sum_{t=1}^{n}K_{h_{n}}(x_{t}-x)\int K^{2}\right]^{1/2}}\wkc N[0,1]\label{eq:gausslim}
\end{equation}
when $\rho=\rho_{n}$, for all drifting sequences $\{\rho_{n}\}\subset\Rho$.
Define the following classes of sequences:\label{subsec:division}
\begin{itemize}
\item \emph{stationary}: $\rho_{n}\goesto\rho$ for some $\rho\in[-1+\delta,1)$,
and $\rho_{n}<1$ for all $n$;
\item \emph{mildly integrated}: $\rho_{n}\goesto1$ but $n(\rho_{n}-1)\goesto-\infty$,
and $\rho_{n}<1$ for all $n$; and
\item \emph{local to unity}: $\rho_{n}\goesto1$, and $n(\rho_{n}-1)\goesto c$
for some $c\leq0$;
\end{itemize}
and let $\rseqs$ denote the collection of all such sequences $\{\rho_{n}\}$.
Though $\rseqs$ is evidently a strict subset of all sequences in
$\Rho$, by an argument given in the proof of \propref{size}, the
convergence \eqref{gausslim} must hold for \emph{all} sequences in
$\Rho$ if it holds for all those in $\rseqs$ (here we adapt the
proof of Lemma~2.1 in \citealp{AC12Ecta}).

It then remains to prove that \eqref{gausslim} holds for stationary,
mildly integrated, and local-to-unity sequences $\{\rho_{n}\}$. In
all cases, the numerator of \eqref{gausslim} is a martingale, and
so is in principle amenable to the application of existing martingale
central limit theory. The main difficulty is to show that the conditional
variance $\sigma_{u}^{2}\sum_{t=1}^{n}K_{h_{n}}^{2}(x_{t}-x)$ converges
weakly to an a.s.\ nonzero limit upon standardisation. This follows
by an application of \thmref{fidi} below. Convergence results of
this kind are available in the literature when $\{\rho_{n}\}$ is
stationary or local to unity, but the proof of this convergence when
$\{\rho_{n}\}$ is mildly integrated requires some genuinely new limit
theory for kernel density estimators, which is the principal contribution
of the following section.

\section{Density estimation: a unified limit theory\label{sec:spatialdensity}}

Our remaining objective is thus to provide some new results on the
asymptotics of functionals of the form $\sum_{t=1}^{n}f_{h_{n}}(x_{t}-x)$
\textendash{} where $f$ is an integrable function and $f_{h}(x)\defeq h^{-1}f(h^{-1}x)$
\textendash{} in the case where $\{x_{t}\}$ is mildly integrated,
i.e.\ when $\rho_{n}\goesto1$ but $n(\rho_{n}-1)\goesto-\infty$.
We shall do this by means of an extension to Theorem~2.1 in \citet[hereafter WP]{WP09ET},
which is stated as \thmref{WPgen} below. An application of this result
to mildly integrated processes, in conjunction with existing results
for local-to-unity and stationary processes, gives the asymptotics
of $\sum_{t=1}^{n}f_{h_{n}}(x_{t}-x)$ for all three classes of processes
considered in the preceding section, which are collected in \thmref{fidi}
below.

\subsection{A general framework\label{subsec:densprelims}}

In order to provide our extension of Theorem~2.1 in \citet[hereafter WP]{WP09ET},
we first restate their assumptions, some of which will also be needed
here. Let $\{\tilde{x}_{n,t}\}_{t=1}^{n}$ be a triangular array,
$\{\tilde{\filt}_{n,t}\}_{t=1}^{n}$ a collection of $\sigma$-fields
such that each $\tilde{x}_{n,t}$ is $\tilde{\filt}_{n,t}$-measurable,
$f:\reals\setmap\reals$, and define 
\[
\Omega_{n}(\eta)\defeq\{(s,t)\mid\eta n\leq s\leq(1-\eta)n\sep s+\eta n\leq t\leq n\}
\]
for $\eta\in(0,1)$. Let $L^{p}$ denote the class of Lebesgue $p$-integrable
functions on $\reals$.

\setcounter{assumption}{613}
\begin{assumption}[Ass.\ 2.1\textendash 2.3 in \citealp{WP09ET}]
~\label{ass:WP}

\begin{enumerate}[{label=\textnormal{\ref{ass:WP}\smaller[0.76]{\arabic*}},leftmargin=1.5cm}]
\item \label{enu:WP:g}$f\in L^{1}\intsect L^{2}$.
\item \label{enu:WP:wkc}There exists a stochastic process $X(r)$ on $[0,1]$
having continuous local time $\loctime_{X}(r,a)$ such that $\tilde{x}_{n,\smlfloor{nr}}\wkc X(r)$
in $\ell_{\infty}([0,1])$.
\item \label{enu:WP:array}There exists an $n_{0}\in\naturals$ such that
for all $0\leq s<t\leq n$ and $n\geq n_{0}$,\footnote{Note that \WP{} have $n_{0}=1$ in their statement of this condition,
but it is clearly sufficient for their result that this condition
hold only for $n$ sufficiently large.} there are constants $\{d_{n,s,t}\}$ such that

\begin{enumerate}
\item for some $m_{0}>0$ and $C>0$, $\inf_{(s,t)\in\Omega_{n}(\eta)}d_{n,s,t}\geq\eta^{m_{0}}/C$
as $n\goesto\infty$, and

\begin{enumerate}
\item $\lim_{\eta\goesto0}\lim_{n\goesto\infty}\frac{1}{n}\sum_{t=(1-\eta)n}^{n}d_{n,0,t}^{-1}=0$,
\item $\lim_{\eta\goesto0}\lim_{n\goesto\infty}\frac{1}{n}\max_{0\leq s\leq(1-\eta)n}\sum_{t=s+1}^{s+\eta n}d_{n,s,t}^{-1}=0$,
\item $\limsup_{n\goesto\infty}\frac{1}{n}\max_{0\leq s\leq n-1}\sum_{t=s+1}^{n}d_{n,s,t}^{-1}<\infty;$
\end{enumerate}
\item conditional on $\tilde{\filt}_{n,s}$, $(\tilde{x}_{n,t}-\tilde{x}_{n,s})/d_{n,s,t}$
has a density $h_{n,s,t}(x)$ which is uniformly bounded (in $n$,
$s$ and $t$) by a constant $K<\infty$, and
\begin{equation}
\lim_{\delta\goesto0}\lim_{n\goesto\infty}\sup_{(s,t)\in\Omega_{n}(\delta^{1/2m_{0}})}\sup_{\smlabs u\leq\delta}\smlabs{h_{n,s,t}(u)-h_{n,s,t}(0)}=0.\label{eq:hequic}
\end{equation}
\end{enumerate}
\end{enumerate}
\end{assumption}

It is evident from \citet{Jeg04} that \enuref{WP:wkc} may be weakened
to finite dimensional convergence (i.e.\ $\tilde{x}_{n,\smlfloor{nr}}\fdd X(r)$)
if $\{\tilde{x}_{n,\smlfloor{nr}}\}$ satisfies the following weak
asymptotic `equicontinuity in probability' condition: that for every
$\epsilon>0$,
\begin{equation}
\lim_{\delta\goesto0}\limsup_{n\goesto\infty}\sup_{\smlabs{r_{1}-r_{2}}\leq\delta}\Prob\{\smlabs{\tilde{x}_{n,\smlfloor{nr_{1}}}-\tilde{x}_{n,\smlfloor{nr_{2}}}}>\epsilon\}=0.\label{eq:ucprob}
\end{equation}
This is considerably weaker than asymptotic equicontinuity (tightness),
which would require control over $\sup_{\smlabs{r_{1}-r_{2}}\leq\delta}\smlabs{\tilde{x}_{n,\smlfloor{nr_{1}}}-\tilde{x}_{n,\smlfloor{nr_{2}}}}$
(and which is of course implied by \enuref{WP:wkc}). However, as
discussed further in \remref{fidis} below, when $\{\tilde{x}_{n,t}\}$
is derived from a mildly integrated process, even such an apparently
weak requirement as \eqref{ucprob} fails to hold: though the finite-dimensional
limit of $\tilde{x}_{n,[nr]}$ exists, it is not separable. However,
it is possible in this case to verify the following strictly weaker
condition, which turns out to be sufficient for the purposes of \thmref{WPgen}
below.

\addtocounter{assumption}{-1}
\begin{assumption}[continued]
~

\begin{enumerate}[{start=2, label=\textnormal{\ref{ass:WP}\smaller[0.76]{\arabic*$^\prime$}},leftmargin=1.5cm}]
\item \label{enu:WPdens} There exists a stochastic process $\tilde{\genloc}:[0,1]\times\reals\setmap\reals_{+}$,
which is continuous a.s.\ with $\int_{\reals}\tilde{\genloc}(r,x)\diff x<\infty$
for all $r\in[0,1]$, such that for every bounded and Lipschitz $g:\reals\setmap\reals$,
\begin{equation}
\frac{1}{n}\sum_{t=1}^{\smlfloor{nr}}g(\tilde{x}_{n,t}-a)\fdd\int_{\reals}g(x-a)\tilde{\genloc}(r,x)\diff x,\label{eq:genCMT}
\end{equation}
over $(r,a)\in[0,1]\times\reals$.
\end{enumerate}
\end{assumption}
Replacing \enuref{WP:wkc} by \enuref{WPdens}, we thus have the following
extension of \WP{}'s Theorem~2.1. The proof appears in \appref{fidiproof}.
\begin{thm}
\label{thm:WPgen}Suppose \enuref{WP:g}, \enuref{WPdens} and \enuref{WP:array}
hold. Then if $\tilde{c}_{n}\goesto\infty$ and $\tilde{c}_{n}/n\goesto0$
\begin{equation}
\frac{\tilde{c}_{n}}{n}\sum_{t=1}^{\smlfloor{nr}}f[\tilde{c}_{n}(\tilde{x}_{n,t}-a)]\fdd\tilde{\genloc}(r,a)\int_{\reals}f\label{eq:loccvg}
\end{equation}
over $(r,a)\in[0,1]\times\reals$.
\end{thm}

\subsection{Application to mildly integrated processes}

\thmref{WPgen} is broad enough to cover the entire class of regressor
processes contemplated in \assref{reg}, even when $\rho=\rho_{n}$
varies with $n$. Indeed, it is the manner in which $\rho_{n}$ approaches
unity (if at all) that determines the density $\tilde{\genloc}$ appearing
in \eqref{genCMT}. In accordance with the division of the sequences
$\{\rho_{n}\}\in\rseqs$ given in \subsecref{division} above, define
\begin{equation}
\genloc(r,a;\{\rho_{n}\})\defeq\begin{cases}
r\sigma_{\rho}\limdens_{\rho}(\sigma_{\rho}a) & \text{if }\{\rho_{n}\}\text{ is stationary}\\
r\gauss(a) & \text{if }\{\rho_{n}\}\text{ is mildly integrated}\\
\loctime_{c}(r,a) & \text{if }\{\rho_{n}\}\text{ is local to unity}
\end{cases}\label{eq:trichotomy}
\end{equation}
where $\limdens_{\rho}$ is the density corresponding to the stationary
solution to \eqref{regproc}, which has variance $\sigma_{\rho}^{2}$;
$\gauss$ is the standard normal density; and $\loctime_{c}(r,a)$
is the local time density (at time $r\in[0,1]$ and point $a\in\reals$)
associated with the normalised Ornstein\textendash Uhlenbeck process,
\begin{equation}
J_{c}(r)\defeq\left(\int_{0}^{1}\e^{2(1-s)c}\diff s\right)^{-1/2}\int_{0}^{r}\e^{(r-s)c}\diff W(s),\label{eq:OU}
\end{equation}
for $W$ a standard Brownian motion on $[0,1]$.

Our main result on the finite-dimensional convergence of density estimators,
when applied to a series $\{x_{t}\}$ satisfying \assref{reg}, may
be stated as follows. Let $\{h_{n}\}$ denote a deterministic, nonzero
bandwidth sequence, define $\stdseq_{n}\defeq\var(x_{n})^{1/2}$,
and recall $f_{h}(x)\defeq h^{-1}f(h^{-1}x)$. The proof appears in
\appref{fidiproof}.\footnote{In an earlier version of this paper (arXiv:1509.05017v3) we showed
that the finite dimensional convergence in \thmref{fidi} may be strengthened
to weak convergence.}
\begin{thm}
\label{thm:fidi}Suppose \assref{reg} holds with $\rho=\rho_{n}$
for some $\{\rho_{n}\}\in\rseqs$, and $f\in L^{1}\intsect L^{2}$.
Then if $h_{n}=o(\stdseq_{n})$ and $n\stdseq_{n}^{-1}h_{n}\goesto\infty$,
\begin{equation}
\frac{\stdseq_{n}}{n}\sum_{t=1}^{\smlfloor{nr}}f_{h_{n}}(x_{t}-\stdseq_{n}a)\fdd\genloc(r,a;\{\rho_{n}\})\int_{\reals}f,\label{eq:fidi}
\end{equation}
over $(r,a)\in[0,1]\times\reals$.
\end{thm}
\begin{rem}
\label{rem:cvgtozero} $\stdseq_{n}\goesto\infty$ whenever $\{\rho_{n}\}$
is mildly integrated or local to unity, and so the arguments given
in the proof of \thmref{fidi} also imply that, in these cases,
\[
\frac{\stdseq_{n}}{n}\sum_{t=1}^{\smlfloor{nr}}f_{h_{n}}(x_{t}-x)\wkc\genloc(r,0;\{\rho_{n}\})\int_{\reals}f
\]
for each $x\in\reals$.
\end{rem}

\begin{rem}
The stationary and local-to-unity cases are covered by the results
of \citet{WM02AS}, \citet{WP09Ecta} and \citet{WHH10SP}. The proof
under mild integration is new to the literature, though the arguments
employed are a combination of those appropriate to the stationary
and local-to-unity cases. Our strategy is to use a kind of law of
large numbers to establish \eqref{genCMT} for the scale-normalised
array
\begin{equation}
\tilde{x}_{n,t}\defeq\var(x_{n})^{-1/2}x_{t}=\stdseq_{n}^{-1}x_{t},\label{eq:tildearray}
\end{equation}
(see \propref{scaledLLN} in \appref{fidiproof}), whence it follows
that $\{\tilde{x}_{n,t}\}$ satisfies \enuref{WPdens}. Since \enuref{WP:g}
and \enuref{WP:array} also hold, it is then possible to invoke \thmref{WPgen}.
\end{rem}

\begin{rem}
\label{rem:fidis} The tripartite classification in \eqref{trichotomy}
is reflected in the different possible finite-dimensional limits $X(r;\{\rho_{n}\})$
of the standardised regressor process $X_{n}(r)\defeq\stdseq_{n}^{-1}x_{\smlfloor{nr}}$.
Under both stationarity and mild integration, the relatively weak
dependence between $X_{n}(r_{1})$ and $X_{n}(r_{2})$ vanishes in
the limit, and so $X$ has the property that $X(r_{1})$ and $X(r_{2})$
are independent for every $r_{1}\neq r_{2}$. This explains why even
such an apparently mild equicontinuity requirement as \eqref{ucprob}
is unavailing for the purposes of proving \thmref{fidi}.

Under mild integration, $\stdseq_{n}\goesto\infty$ and an invariance
principle operates to ensure that the marginals of $X(r)$ are standard
normal; whereas in the stationary case, $\stdseq_{n}$ is bounded
and the marginals have density $\limdens_{\rho}$, which depends on
the distribution of $\{\err_{t}\}$. The limiting process $X$ under
mild integration thus corresponds to a continuous-time, standard normal
white noise process. (A rigorous basis for these assertions is provided
by \propref{scaledLLN}\enuref{LLNp2} in \appref{fidiproof}, and
the proof thereof.)

The strong dependence between $X_{n}(r_{1})$ and $X_{n}(r_{2})$
that is a characteristic of local-to-unity processes ensures that,
in this case, $X_{n}$ converges weakly to the diffusion $J_{c}$
(see \eqref{OU} above). As $c\goesto-\infty$, the finite-dimensional
distributions of $J_{c}$ converge to those of standard normal white
noise process: and in this sense there is continuity, in the limit,
at the boundary demarcating mildly integrated and local-to-unity processes.
\end{rem}

\section{Conclusion}

This paper has established the asymptotic size of the nonparametric
$t$ test in a predictive regression, when the regressor is possibly
highly persistent. Our work on this problem has necessitated the development
of some new limit theory for kernel density estimators, when these
are applied to mildly integrated processes. These new results have
allowed us to give a unified treatment of kernel density and regression
estimators that encompasses stationary, mildly integrated and local-to-unity
processes.

A notable implication of our results is that conventional nonparametric
inferences, using normal critical values, remain valid regardless
of the degree of persistence of the regressor. This may be counted
a significant advantage of kernel nonparametric estimators over their
parametric counterparts, which partially compensates for their lower
rates of convergence and \textendash{} in the case of integrated regressors
\textendash{} their limited applicability to models with multiple
regressors.

\section{References}

{\vspace{-10pt}\singlespace\bibliographystyle{ecta}
\bibliography{time-series,asymptotics}
}

\appendix

\section{Proof of \propref{size}\label{app:unifproof}}

Throughout the Appendices (excepting Section\ \ref{app:WPproof}),
Assumptions~\assref{reg} and \assref{smoothing} are always maintained,
even when not explicitly referenced.
\begin{notation*}
For $p\in(1,\infty)$ and a function $f:\reals\setmap\reals$, define
$\smlnorm f_{p}\defeq(\int\smlabs{f(x)}^{p}\diff x)^{1/p}$ and $\smlnorm f_{\infty}\defeq\sup_{x\in\reals}\smlabs{f(x)}$;
for a random variable $X$, $\smlnorm X_{p}\defeq(\expect\smlabs X^{p})^{1/p}$,
and $\smlnorm X_{\infty}$ denotes the essential supremum of $X$.
$C$, $C_{1}$, etc., denote generic constants which may take on different
values even at different places in the same proof.
\end{notation*}

We shall need the following auxiliary results, the proofs of which
appear in Section~S.1 of the Online Supplement. Recall the definition
of $\rseqs$, and the classification of the sequences $\{\rho_{n}\}\in\rseqs$
given in \subsecref{unif}. Let 
\[
\scseq_{n}\defeq\scseq_{n}(\{\rho_{n}\})\defeq n\stdseq_{n}^{-1}
\]
where $\stdseq_{n}\defeq\var(x_{n})^{1/2}$ as was defined in \eqref{tildearray}.
\begin{lem}
\label{lem:std}~Suppose $\{\rho_{n}\}\in\rseqs$. Then
\[
n^{1/2}\lesssim\nseq_{n}(\{\rho_{n}\})\lesssim n
\]
\end{lem}
The next lemma is a direct consequence of \thmref{fidi}, and is the
principal implication of that theorem needed for the proof \propref{size}.
For $\{\rho_{n}\}\in\rseqs$, define
\[
\tau(x)\defeq\tau(x,\{\rho_{n}\})\defeq\begin{cases}
\sigma_{\rho}\limdens_{\rho}(x) & \text{if }\{\rho_{n}\}\text{ is stationary}\\
\gauss(0) & \text{if }\{\rho_{n}\}\text{ is mildly integrated}\\
\loctime_{c}(1,0) & \text{if }\{\rho_{n}\}\text{ is local to unity},
\end{cases}
\]
where $\nu_{\rho}$ denotes the density of the stationary solution
to \eqref{regproc} (for $\rho<1$), and $\sigma_{\rho}^{2}$ its
variance.
\begin{lem}
\label{lem:cvg} Suppose $\{\rho_{n}\}\in\rseqs$. Then if $\alpha\geq1$
and $\beta=0$, or $\alpha=1$ and $\beta\in[0,1]$,
\[
\frac{1}{e_{n}}\sum_{t=1}^{n}\frac{1}{h_{n}}K^{\alpha}\left(\frac{x_{t}-x}{h_{n}}\right)\abs{\frac{x_{t}-x}{h_{n}}}^{\beta}\wkc\tau(x)\int_{\reals}K^{\alpha}(u)\smlabs u^{\beta}\deriv u,
\]
where $\tau(x)>0$ a.s.
\end{lem}
\begin{lem}
\label{lem:sig2}For every $x\in\reals$, $\hat{\sigma}_{u}^{2}(x)=\sigma_{u}^{2}+o_{p}(1)$.
\end{lem}

\begin{proof}[Proof of \propref{size}]
 Suppose that we:
\begin{enumerate}
\item show that \eqref{tstatcvg} holds for every $\{\rho_{n}\}\in\rseqs$;
and then
\item deduce from (i) that \eqref{tstatcvg} holds for \emph{all} $\{\rho_{n}\}\subset\Rho$.
\end{enumerate}
The proofs of (i) and (ii) are given immediately below. Now by definition
of the limit supremum, there must exist a $\{\rho_{n}^{\ast}\}\subset\Rho$
such that
\[
\limsup_{n\goesto\infty}\sup_{\rho\in\Rho}\Prob_{\rho}\{\smlabs{\hat{t}_{n}(x)}\geq z_{1-\alpha/2}\}=\lim_{n\goesto\infty}\Prob_{\rho_{n}^{\ast}}\{\smlabs{\hat{t}_{n}(x)}\geq z_{1-\alpha/2}\}.
\]
It follows from (ii) that
\begin{align*}
\lim_{n\goesto\infty}\Prob_{\rho_{n}^{\ast}}\{\smlabs{\hat{t}_{n}(x)}\geq z_{1-\alpha/2}\}=\Prob\{\smlabs{N[0,1]}\geq z_{1-\alpha/2}\}=\alpha
\end{align*}
whence the $t$ test has asymptotic size $\alpha$. Asymptotic similarity
of the $t$ test follows by an analogous argument.

\subparagraph*{(i)}

Let $x\in\reals$ and $\{\rho_{n}\}\in\rseqs$. In view of \lemref{sig2},
straightforward calculations yield that under $\nullhyp$
\begin{equation}
\hat{t}_{n}(x)=[\upsilon_{n}(x)+b_{n}(x)](1+o_{p}(1))\label{eq:tdecomp}
\end{equation}
where
\[
\upsilon_{n}(x)=\frac{h_{n}^{1/2}\sum_{t=1}^{n}K_{h_{n}}(x_{t}-x)u_{t+1}}{\sigma_{u}\left[\sum_{t=1}^{n}K_{h_{n}}(x_{t}-x)\int K^{2}\right]^{1/2}}
\]
is as defined in \eqref{gausslim}, and 
\begin{equation}
b_{n}(x)\defeq\frac{h_{n}^{1/2}\sum_{t=1}^{n}K_{h_{n}}(x_{t}-x)[\regfn(x_{t})-\regfn(x)]}{\sigma_{u}\left(\int_{\reals}K^{2}\sum_{t=1}^{n}K_{h_{n}}(x_{t}-x)\right)^{1/2}}\eqdef\frac{b_{n,1}(x)}{b_{n,2}(x)}.\label{eq:bn}
\end{equation}

By \lemref{cvg} and the fact that $\smlabs{\regfn(x_{t})-\regfn(x)}\leq C\smlabs{x_{t}-x}$
(by \enuref{regfn}),
\begin{equation}
b_{n,1}(x)\leq h_{n}^{3/2}\sum_{t=1}^{n}\frac{1}{h_{n}}K\left(\frac{x_{t}-x}{h_{n}}\right)\abs{\frac{x_{t}-x}{h_{n}}}\lesssim_{p}h_{n}^{3/2}\scseq_{n},\label{eq:bn1}
\end{equation}
and by \lemref{cvg}, 
\begin{equation}
\scseq_{n}^{-1/2}b_{n,2}(x)\wkc\sigma_{u}\left(\tau(x)\int_{\reals}K^{2}\right)^{1/2},\label{eq:bn2}
\end{equation}
which is strictly positive a.s. Together \eqref{bn}\textendash \eqref{bn2}
yield 
\begin{equation}
\smlabs{b_{n}(x)}\isbigop h_{n}^{3/2}\scseq_{n}^{1/2}=o(1)\label{eq:bnnegl}
\end{equation}
since $h_{n}=o(n^{-1/3})$ by assumption, and $\scseq_{n}\lesssim n$
by \lemref{std}.\noeqref{eq:bn1}

The limiting distribution of $\upsilon_{n}(x)$ may be obtained via
an application of an appropriate martingale CLT. Consider the closely
related quantity
\begin{equation}
M_{n}\defeq\left(\frac{h_{n}}{\scseq_{n}}\right)^{1/2}\sum_{t=1}^{n}K_{h_{n}}(x_{t}-x)u_{t+1}.\label{eq:MGforCLT}
\end{equation}
Under \enuref{regdist}, the summands $\{K_{h_{n}}(x_{t-1}-x)u_{t}\}$
are adapted to $\filtg_{t}=\sigma(\{x_{s,}u_{s}\}_{s\leq t})$, with
\[
\expect[K_{h_{n}}(x_{t-1}-x)u_{t}\mid\filt_{t-1}]=K_{h_{n}}(x_{t-1}-x)\cdot\expect[u_{t}\mid\filt_{t-1}]=0.
\]
Hence $M_{n}$ is the row sum of a martingale difference array, with
conditional variance
\begin{align}
\smlcv{M_{n}} & =\frac{\sigma_{u}^{2}}{\scseq_{n}h_{n}}\sum_{t=1}^{n}K^{2}\left(\frac{x_{t}-x}{h_{n}}\right)\wkc\sigma_{u}^{2}\tau(x)\int_{\reals}K^{2},\label{eq:cvlimit}
\end{align}
by \lemref{cvg}. Furthermore, the (standardised) summands in \eqref{MGforCLT}
satisfy a conditional Lyapunov condition, since
\begin{align}
 & \sum_{t=1}^{n}\expect\left[\left\{ \left(\frac{h_{n}}{\scseq_{n}}\right)^{1/2}K_{h_{n}}(x_{t}-x)u_{t+1}\right\} ^{4}\mid\filtg_{t}\right]\nonumber \\
 & \qquad\qquad\qquad=\frac{1}{(\scseq_{n}h_{n})^{2}}\sum_{t=1}^{n}K^{4}\left(\frac{x_{t}-x}{h_{n}}\right)\expect[\smlabs{u_{t+1}}^{4}\mid\filtg_{t}]\nonumber \\
 & \qquad\qquad\qquad\leq\frac{C}{\scseq_{n}h_{n}}\cdot\frac{1}{\scseq_{n}h_{n}}\sum_{t=1}^{n}K^{4}\left(\frac{x_{t}-x}{h_{n}}\right)\nonumber \\
 & \qquad\qquad\qquad\isbigop\frac{1}{\scseq_{n}h_{n}}=o(1)\label{eq:lyapunov}
\end{align}
by \enuref{regdist}, \lemref{cvg} and the fact that $\scseq_{n}h_{n}\gtrsim n^{1/2}h_{n}\goesto\infty$
(by \enuref{smoothing:hn} and \lemref{std}).

When $\{\rho_{n}\}$ is stationary or mildly integrated, the r.h.s.\ of
\eqref{cvlimit} is non-random, and so the asymptotic normality of
\eqref{MGforCLT} follows from Theorem~3.2 in \citet{HH80}: the
relevant conditions having been verified by \eqref{cvlimit} and \eqref{lyapunov}.
Thus in both cases
\begin{equation}
M_{n}\defeq\left(\frac{h_{n}}{\scseq_{n}}\right)^{1/2}\sum_{t=1}^{n}K_{h_{n}}(x_{t}-x)u_{t+1}\wkc\sigma_{u}\left(\tau(x)\int_{\reals}K^{2}\right)^{1/2}\cdot\xi,\label{eq:statmiconv}
\end{equation}
where $\tau(x)>0$ is a constant.

When $\{\rho_{n}\}$ is local to unity, $\tau(x)=\loctime_{c}(1,0)$
is a random local time density, and we must instead apply Theorem~2.1
of \citet{Wang14ET}. This requires that we additionally verify the
stronger conditions of that theorem. Under \assref{reg}, it is easy
to see that $\{(\err_{t},u_{t}),\filtg_{t}\}$ satisfy his Assumption~1.
That his Assumption~2 is satisfied follows from
\[
\max_{t\leq n}\abs{\left(\frac{h_{n}}{\scseq_{n}}\right)^{1/2}K_{h_{n}}(x_{t}-x)}\leq\frac{1}{(\scseq_{n}h_{n})^{1/2}}\sup_{x\in\reals}\smlabs{K(x)}=o(1)
\]
and
\begin{multline*}
\left(\frac{h_{n}}{n\scseq_{n}}\right)^{1/2}\sum_{t=1}^{n}K_{h_{n}}(x_{t}-x)\smlabs{\expect_{t}\err_{t+1}u_{t+1}}\\
\leq\sigma_{u}\left(\frac{h_{n}}{n\scseq_{n}}\right)^{1/2}\sum_{t=1}^{n}K_{h_{n}}(x_{t}-x)\lesssim_{p}\left(\frac{h_{n}^{3/2}\scseq_{n}}{n}\right)^{1/2}\lesssim\left(\frac{\scseq_{n}}{n}\right)^{1/2}=o(1),
\end{multline*}
which follows by the Cauchy-Schwarz inequality and \lemref{cvg}.
Finally, for the purposes of verifying his Assumption~3, we note
that
\begin{align}
\frac{1}{\scseq_{n}h_{n}}\sum_{t=1}^{n}K^{2}\left(\frac{x_{t}-x}{h_{n}}\right)u_{t+1}^{2} & =\smlcv{M_{n}}+o_{p}(1)\wkc\tau(x)\int_{\reals}K^{2}=\loctime_{c}(1,0)\int_{\reals}K^{2}\label{eq:ssqconv}
\end{align}
by Theorem~2.23 in \citet{HH80} and \eqref{cvlimit}. $\loctime_{c}$
is the local time density of the process $J_{c}$ given in \eqref{OU},
and is thus a functional of the standard Brownian motion $W$ that
emerges as the weak limit of $n^{-1/2}\sum_{t=1}^{\smlfloor{nr}}\err_{t}\wkc W(r)$.
This convergence holds jointly with \eqref{ssqconv}, and so \citeauthor{Wang14ET}'s
Assumption~3 is satisfied. It therefore follows by Theorem~2.1 of
\citet{Wang14ET} that
\begin{equation}
M_{n}\defeq\left(\frac{h_{n}}{\scseq_{n}}\right)^{1/2}\sum_{t=1}^{n}K_{h_{n}}(x_{t}-x)u_{t+1}\wkc\sigma_{u}\left(\loctime_{c}(1,0)\int_{\reals}K^{2}\right)^{1/2}\cdot\xi\label{eq:luconv}
\end{equation}
holds jointly with \eqref{cvlimit}, where $\xi\sim N[0,1]$ is independent
of $\loctime_{c}(1,0)$.

Finally, deduce from \eqref{tdecomp}, \eqref{bnnegl}, \eqref{cvlimit},
\eqref{statmiconv} and \eqref{luconv} that
\[
\hat{t}_{n}(x)=\upsilon_{n}(x)+o_{p}(1)=\frac{M_{n}}{\smlcv{M_{n}}}+o_{p}(1)\wkc\frac{\sigma_{u}\left(\tau(x)\int_{\reals}K^{2}\right)^{1/2}\cdot\xi}{\sigma_{u}\left(\tau(x)\int_{\reals}K^{2}\right)^{1/2}}=\xi\sim N[0,1]
\]
for all $\{\rho_{n}\}\in\rseqs$.

\subparagraph*{(ii).}

The argument here largely follows the proof of Lemma~2.1 in \citet{AC12Ecta}.
Let $f$ be an arbitrary bounded and Lipschitz function. It follows
from part~(i) of the proof that
\begin{equation}
\expect_{\rho_{n}}f(\hat{t}_{n})\goesto\expect f(\xi)\label{eq:wkccrit}
\end{equation}
for every $\{\rho_{n}\}\in\rseqs$, where $\xi\sim N[0,1]$ and where
$\expect_{\rho_{n}}$ is indexed by the true parameters $\rho_{n}$.
We need to show that the preceding holds for every $\{\rho_{n}\}\subset\Rho$:
i.e.\ that it holds for \emph{all} sequences, not merely those in
$\rseqs$. To that end, let $\{\rho_{n}\}\subset\Rho$ be given. To
prove \eqref{wkccrit}, it suffices to show that for every subsequence
$\{p_{n}\}$ of $\{n\}$, there exists a further subsequence $\{w_{n}\}$
of $\{p_{n}\}$ such that
\begin{equation}
\expect_{\rho_{w_{n}}}f(\hat{t}_{w_{n}})\goesto\expect f(\xi).\label{eq:wkcsub}
\end{equation}

Let $\{p_{n}\}$ be an arbitrary subsequence of $\{n\}$, and $c_{n}\defeq n(\rho_{n}-1)$.
By a compactification of $\reals$, $\{(\rho_{p_{n}},c_{p_{n}})\}$
has an accumulation point $(\abv{\rho},\abv c)\in\Rho\times[-\infty,0]$.
Now let $\{w_{n}\}$ be a subsequence of $\{p_{n}\}$, chosen as follows.
If
\begin{enumerate}
\item $\abv{\rho}<1$: choose $\{w_{n}\}$ such that $\rho_{w_{n}}\goesto\abv{\rho}$
and $\rho_{w_{n}}<1$, for all $n\in\naturals$;
\item $\abv{\rho}=1$ and either:
\begin{enumerate}
\item $\abv c\in(-\infty,0]$: choose $\{w_{n}\}$ such that $c_{w_{n}}\goesto\abv c$;
or
\item $\abv c=-\infty$: choose $\{w_{n}\}$ such that $(\rho_{w_{n}},c_{w_{n}})\goesto(1,-\infty)$.
\end{enumerate}
\end{enumerate}
Note that in case (ii)(b),
\begin{equation}
w_{n}^{-1}c_{w_{n}}=\rho_{w_{n}}-1\goesto0\label{eq:wncwn}
\end{equation}
as $n\goesto\infty$.

Corresponding to the three cases above, construct a new sequence $\{\rho_{n}^{\prime}\}$
as follows.
\begin{enumerate}
\item $\rho_{n}^{\prime}=\rho_{w_{k}}$ for $w_{k}\leq n<w_{k+1}$: then
$\rho_{n}^{\prime}\goesto\abv{\rho}<1$, whence $\{\rho_{n}^{\prime}\}$
is a stationary sequence.
\item $\rho_{n}^{\prime}=1+n^{-1}c_{w_{k}}$ for $w_{k}\leq n<w_{k+1}$.
Then by construction,
\[
c_{n}^{\prime}\defeq n(\rho_{n}^{\prime}-1)=c_{w_{k}}\qquad\text{for }w_{k}\leq n<w_{k+1},
\]
and hence in case:
\begin{enumerate}
\item $\lim_{n\goesto\infty}c_{n}^{\prime}=\lim_{k\goesto\infty}c_{w_{k}}=\text{\ensuremath{\abv c\in}}(-\infty,0]$,
so $\{\rho_{n}^{\prime}\}$ is a local-to-unity sequence;
\item $\lim_{n\goesto\infty}c_{n}^{\prime}=\lim_{k\goesto\infty}c_{w_{k}}=-\infty$,
and for $w_{k}\leq n<w_{k+1}$,
\[
\smlabs{\rho_{n}^{\prime}-1}=n^{-1}\smlabs{c_{n}^{\prime}}=n^{-1}\smlabs{c_{w_{k}}}\leq w_{k}^{-1}\smlabs{c_{w_{k}}}\goesto0
\]
as $k\goesto\infty$, by \eqref{wncwn}. Thus $\rho_{n}^{\prime}\goesto1$
and $\{\rho_{n}^{\prime}\}$ is a mildly integrated sequence.
\end{enumerate}
\end{enumerate}
It follows that $\{\rho_{n}^{\prime}\}\in\rseqs$ in all cases, and
thus \eqref{wkccrit} holds for $\{\rho_{n}^{\prime}\}$ by part~(i)
of the proof. Since by construction $\rho_{w_{n}}^{\prime}=\rho_{w_{n}}$
for all $n\in\naturals$, we finally have
\[
\expect f(\xi)=\lim_{n\goesto\infty}\expect_{\rho_{n}^{\prime}}f(\hat{t}_{n})=\lim_{n\goesto\infty}\expect_{\rho_{w_{n}}^{\prime}}f(\hat{t}_{w_{n}})=\lim_{n\goesto\infty}\expect_{\rho_{w_{n}}^{\prime}}f(\hat{t}_{w_{n}})
\]
and thus \eqref{wkcsub} holds.
\end{proof}

\section{Proofs of Theorems~\ref{thm:WPgen} and \ref{thm:fidi}\label{app:fidiproof}}

\subsection{Proof of \thmref{WPgen}\label{app:WPproof}}

Similarly to the proof of Theorem~2.1 in \citet{WP09ET}, define
\begin{align*}
L_{n}(r,a) & \defeq\frac{\tilde{c}_{n}}{n}\sum_{k=1}^{\smlfloor{nr}}f[\tilde{c}_{n}(\tilde{x}_{k,n}-a)]\\
L_{n,\epsilon}(r,a) & \defeq\frac{\tilde{c}_{n}}{n}\sum_{k=1}^{\smlfloor{nr}}\int_{\reals}f[\tilde{c}_{n}(\tilde{x}_{k,n}-a+z\epsilon)]\gauss(z)\diff z,
\end{align*}
and set $\gauss_{\epsilon}(x)\defeq\epsilon^{-1}\gauss(\epsilon^{-1}x)$.
It follows from Lemma~7 in \citet{Jeg04} that, for each $\epsilon>0$
fixed,
\[
\lim_{n\goesto\infty}\biggabs{L_{n,\epsilon}(r,a)-\frac{1}{n}\sum_{k=1}^{\smlfloor{nr}}\gauss_{\epsilon}(\tilde{x}_{k,n}-a)\int_{\reals}f}=0.
\]
Furthermore, the arguments used by \citet{WP09ET} to prove that 
\[
\lim_{\epsilon\goesto0}\lim_{n\goesto\infty}\expect\smlabs{L_{n}(r,a)-L_{n,\epsilon}(r,a)}=0,
\]
for each $a\in\reals$, which corresponds to (5.1) in that paper,
require only their Assumptions~2.1 and 2.3, both of which are maintained
here (as \enuref{WP:g} and \enuref{WP:array} respectively). Finally,
by \enuref{WPdens},
\begin{multline*}
\frac{1}{n}\sum_{k=1}^{\smlfloor{nr}}\gauss_{\epsilon}(\tilde{x}_{k,n}-a)\fdd\int_{\reals}\gauss_{\epsilon}(x-a)\tilde{\genloc}(r,x)\diff x\\
=\int_{\reals}\gauss(x)\tilde{\genloc}(r,\epsilon x+a)\diff x=\tilde{\genloc}(r,a)+o_{p}(1)
\end{multline*}
over $(r,a)\in[0,1]\times\reals$ as $n\goesto\infty$ and then $\epsilon\goesto0$,
since $\tilde{\genloc}$ is continuous a.s.\hfill\qedsymbol{}

\subsection{Proof of \thmref{fidi}\label{subsec:fidiproof}}

We separately consider $\{\rho_{n}\}\in\rseqs$ that are local to
unity, mildly integrated, and stationary.

\paragraph*{$\{\rho_{n}\}$ local to unity.}

Proposition~7.1 in \citet{WP09Ecta}, together with the arguments
used to prove their Proposition~7.2, establish that $\{\tilde{x}_{n,t}\}$
satisfies \enuref{WP:wkc} and \enuref{WP:array}. Thus, in this case,
the result follows by \thmref{WPgen}.

\paragraph*{$\{\rho_{n}\}$ mildly integrated.}

In this case, we shall need the following two results, the proofs
of which are given in \appref{mild}. Recall the definition of $\tilde{x}_{n,t}\defeq d_{n}^{-1}x_{t}$
given in \eqref{tildearray} above.
\begin{prop}
\label{prop:scaledLLN} Suppose $g$ is bounded and Lipschitz, and
$\{\rho_{n}\}$ is mildly integrated. Then
\begin{enumerate}
\item \label{enu:LLNp1}$\frac{1}{n}\sum_{t=1}^{\smlfloor{nr}}g(\tilde{x}_{n,t})=\frac{1}{n}\sum_{t=1}^{\smlfloor{nr}}\expect g(\tilde{x}_{n,t})+o_{p}(1)$;
and
\item \label{enu:LLNp2}$\frac{1}{n}\sum_{t=1}^{\smlfloor{nr}}\expect g(\tilde{x}_{n,t})\goesto r\int_{\reals}g(x)\gauss(x)\diff x$.
\end{enumerate}
\end{prop}

\begin{prop}
\label{prop:WP3} Suppose $\{\rho_{n}\}$ is mildly integrated. Then
$\tilde{x}_{n,t}$ satisfies \enuref{WP:array} with $\tilde{\filt}_{n,t}\defeq\sigma(\{\err_{s}\}_{s\leq t})$.
\end{prop}

It follows immediately from \propref{scaledLLN} that for every $g$
bounded and Lipschitz,
\[
\frac{1}{n}\sum_{t=1}^{\smlfloor{nr}}g(\tilde{x}_{n,t}-a)=\frac{1}{n}\sum_{t=1}^{\smlfloor{nr}}\expect g(\tilde{x}_{n,t}-a)+o_{p}(1)\inprob r\int_{\reals}g(x-a)\gauss(x)\diff x
\]
for each $(r,a)\in[0,1]\times\reals$. Thus \enuref{WPdens} holds
with $\tilde{\genloc}(r,a)=r\gauss(a)$. By \propref{WP3}, $\{\tilde{x}_{n,t}\}$
satisfies \enuref{WP:array}, whence the result follows by \thmref{WPgen}.

\paragraph*{$\{\rho_{n}\}$ stationary.}

Since $\stdseq_{n}\lesssim1$ in this case, it follows from Theorem~1
in \citet{WHH10SP}, with minor modifications, that 
\[
\frac{\stdseq_{n}}{n}\sum_{t=1}^{\smlfloor{nr}}f_{h_{n}}(x_{t}-\stdseq_{n}a)=\frac{\stdseq_{n}}{n}\sum_{t=1}^{\smlfloor{nr}}\expect f_{h_{n}}(x_{t}-\stdseq_{n}a)+o_{p}(1).
\]

It remains to determine the limit of the r.h.s. To that end, let $\nu_{\rho,t}$
and $\psi_{\rho,t}$ respectively denote the probability density and
characteristic function of $x_{t}$, and $\nu_{\rho}$ and $\psi_{\rho}$
those of the stationary solution to \eqref{regproc}, for $\rho<1$.
Let $t_{n}\in\naturals$ with $t_{n}\leq n$ and $t_{n}\goesto\infty$.
Since $\rho_{n}\goesto\rho<1$ is bounded away from unity, we have
\begin{equation}
x_{t_{n}}=\sum_{s=0}^{t_{n}-1}\rho_{n}^{s}v_{t_{n}-s}\eqdist\sum_{s=0}^{t_{n}-1}\rho_{n}^{s}v_{-s}\inprob\sum_{s=0}^{\infty}\rho^{s}v_{-s}\label{eq:xtnweak}
\end{equation}
where the r.h.s.\ has density $\nu_{\rho}$. Deduce $\psi_{\rho_{n},t_{n}}(\lambda)\goesto\psi_{\rho}(\lambda)$
for each $\lambda\in\reals$, whence
\begin{align}
\smlnorm{\nu_{\rho_{n},t_{n}}-\nu_{\rho}}_{\infty} & \leq\int_{\{\smlabs{\lambda}\leq A\}}\smlabs{\psi_{\rho_{n},t_{n}}(\lambda)-\psi_{\rho}(\lambda)}\diff\lambda+\int_{\{\smlabs{\lambda}>A\}}[\smlabs{\psi_{\rho_{n},t_{n}}(\lambda)}+\smlabs{\psi_{\rho}(\lambda)}]\diff\lambda\nonumber \\
 & \goesto0,\label{eq:denscvg}
\end{align}
as $n\goesto\infty$ and then $A\goesto\infty$, where we have used
$\smlabs{\psi_{\rho_{n},t_{n}}(\lambda)}\pmax\smlabs{\psi_{\rho_{n}}(\lambda)}\leq\smlabs{\psi_{\err}(\phi_{0}\lambda)}$
to control the integral over $\{\smlabs{\lambda}>A\}$.

Since the convergence in \eqref{xtnweak} also holds in mean square,
taking $t_{n}=n$ yields $\stdseq_{n}=\var(x_{n})^{1/2}\goesto\sigma_{\rho}$,
the standard deviation associated to the density $\nu_{\rho}$. Thus
by \eqref{denscvg}
\begin{multline*}
\expect f_{h_{n}}(x_{t_{n}}-\stdseq_{n}a)=\int_{\reals}f(x)\nu_{\rho_{n},t_{n}}(\stdseq_{n}a+h_{n}x)\diff x\\
=\int_{\reals}f(x)\nu_{\rho}(\stdseq_{n}a+h_{n}x)\diff x+o(1)\goesto\nu_{\rho}(\sigma_{\rho}a)\int_{\reals}f,
\end{multline*}
and hence
\[
\frac{\stdseq_{n}}{n}\sum_{t=\smlfloor{n\delta}+1}^{\smlfloor{nr}}\expect f_{h_{n}}(x_{t}-\stdseq_{n}a)\goesto(r-\delta)\sigma_{\rho}\nu_{\rho}(\sigma_{\rho}a)\int_{\reals}f\goesto r\sigma_{\rho}\nu_{\rho}(\sigma_{\rho}a)\int_{\reals}f
\]
as $n\goesto\infty$ and then $\delta\goesto0$, while
\begin{multline*}
\abs{\frac{\stdseq_{n}}{n}\sum_{t=1}^{\smlfloor{n\delta}}\expect f_{h_{n}}(x_{t}-\stdseq_{n}a)}\\
\leq\delta\cdot\stdseq_{n}\max_{1\leq t\leq\smlfloor{n\delta}}\smlabs{\expect f_{h_{n}}(x_{t}-\stdseq_{n}a)}\leq C\delta\max_{1\leq t\leq\smlfloor{n\delta}}\smlnorm{\nu_{\rho_{n},t}}_{\infty}\smlnorm f_{1}\goesto0
\end{multline*}
as $n\goesto\infty$ and then $\delta\goesto0$, since
\[
\smlnorm{\nu_{\rho_{n},t}}_{\infty}\leq\int_{\reals}\smlabs{\psi_{\rho_{n},t}(\lambda)}\diff\lambda\leq\int_{\reals}\smlabs{\psi_{\err}(\phi_{0}\lambda)}\diff\lambda<\infty.\tag*{\qedsymbol}
\]

\section{Proofs of Propositions \ref{prop:scaledLLN} and \ref{prop:WP3}\label{app:mild}}

\subsection{Preliminaries\label{subsec:prelims}}

Under \assref{reg}, we may write $x_{t}=\sum_{k=0}^{\infty}a_{t,k}\err_{t-k}$,
where 
\begin{equation}
a_{t,k}\defeq a_{t,k}(\rho)\defeq\sum_{l=0}^{k\pmin(t-1)}\rho^{l}\phi_{k-l}.\label{eq:attk}
\end{equation}
Observe that this quantity does not depend on $t$ for $0\leq k\leq t-1$,
and we will accordingly write $a_{k}\defeq a_{t,k}$ in this case.
We shall make frequent use of the decomposition 
\begin{equation}
x_{t}=\sum_{k=0}^{\infty}a_{t,k}\err_{t-k}=\sum_{k=t-s+1}^{\infty}a_{t,k}\err_{t-k}+\sum_{k=0}^{t-s}a_{k}\err_{t-k}\eqdef x_{s-1,t}^{\prime}+x_{s,t},\label{eq:xdecmp1}
\end{equation}
for $s\in\{1,\ldots,t\}$: note that $x_{s-1,t}^{\prime}$ and $x_{s,t}$
are independent.

We shall also need the following lemma, the proof of which appears
in Section~S.2 of the Online Supplement. Recall that $\stdseq_{n}^{2}=\var(x_{n})$
and $\phi=\sum_{k=0}^{\infty}\phi_{k}$.

\begin{lem}
\label{lem:variance} Suppose $\{\rho_{n}\}$ is mildly integrated
and $\epsilon>0$. Then
\begin{enumerate}
\item \label{enu:MInegl}$\rho_{n}^{n}\goesto0$;
\item \label{enu:MI:dn}$\stdseq_{n}^{2}\sim\phi^{2}(1-\rho_{n}^{2})^{-1}$;
and
\item for any sequence $\{t_{n}\}$ with $n\epsilon\leq t_{n}\leq n$,
\[
\var(x_{t_{n}})\sim\var(x_{1,t_{n}})\sim\stdseq_{n}^{2}.
\]
\end{enumerate}
\end{lem}

\subsection{Proof of \propref{scaledLLN}\label{subsec:llnproof}}

We first state and prove the following auxiliary lemma, which is the
key ingredient in the proof of the first part of \propref{scaledLLN}.
For a function $g$ bounded and Lipschitz, let $\smlnorm g_{\Lip}\defeq\sup_{x\neq y}\smlabs{g(x)-g(y)}/\smlabs{x-y}$.
\begin{lem}
\label{lem:Ebound} For any $g$ bounded and Lipschitz,
\begin{equation}
\expect\abs{\sum_{t=1}^{n}[g(x_{t})-\expect g(x_{t})]}\leq\smlnorm g_{\Lip}\sum_{k=0}^{\infty}\left(\sum_{t=1}^{n}a_{t,k}^{2}\right)^{1/2}\leq\smlnorm g_{\Lip}n^{1/2}\frac{\sum_{k=0}^{\infty}\smlabs{\phi_{k}}}{1-\smlabs{\rho}},\label{eq:Ebound}
\end{equation}
where the second inequality holds if $\smlabs{\rho}<1$.
\end{lem}
\begin{proof}
Let $\expect_{t}[\cdot]\defeq\expect[\cdot\mid\filtg_{t}]$. We decompose
\[
g(x_{t})-\expect g(x_{t})=\sum_{k=0}^{\infty}[\expect_{t-k}g(x_{t})-\expect_{(t-1)-k}g(x_{t})]
\]
where the sum on the r.h.s.\ converges a.s.,\ since $\expect_{t-k}g(x_{t})\goesto\expect g(x_{t})$
a.s.\ as $k\goesto\infty$, by the reverse martingale convergence
theorem. Therefore we may write 
\begin{equation}
\sum_{t=1}^{n}[g(x_{t})-\expect g(x_{t})]=\sum_{k=0}^{\infty}\sum_{t=1}^{n}[\expect_{t-k}g(x_{t})-\expect_{(t-1)-k}g(x_{t})]\eqdef\sum_{k=0}^{\infty}M_{n,k}.\label{eq:Qndecomp}
\end{equation}
Clearly, by the orthogonality of martingale differences,
\begin{equation}
\expect M_{n,k}^{2}=\sum_{t=1}^{n}\expect[\expect_{t-k}g(x_{t})-\expect_{(t-1)-k}g(x_{t})]^{2}.\label{eq:Mnk2}
\end{equation}

We have
\begin{alignat*}{5}
x_{t}=\sum_{k=0}^{\infty}a_{t,k}\err_{t-k} & = &  & \sum_{s=0}^{k-1}a_{t,s}\err_{t-s} &  & +a_{t,k}\err_{t-k} &  & +\sum_{s=k+1}^{\infty}a_{t,s}\err_{t-s}\\
 & \eqdist &  & \sum_{s=0}^{k-1}a_{t,s}\err_{t-s} &  & +a_{t,k}\err^{\ast} &  & +\sum_{s=k+1}^{\infty}a_{t,s}\err_{t-s} &  & \eqdef x_{t}^{\ast}
\end{alignat*}
where `$=_{d}$' denotes equality in distribution, and $\err^{\ast}\eqdist\err_{0}$
is defined to be independent of $\{\err_{t}\}$, and hence also of
$\filtg_{t-k}$. Thus $\expect_{(t-1)-k}g(x_{t})=\expect_{t-k}g(x_{t}^{\ast})$,
whence
\[
\smlabs{\expect_{t-k}g(x_{t})-\expect_{(t-1)-k}g(x_{t})}=\smlabs{\expect_{t-k}[g(x_{t})-g(x_{t}^{\ast})]}\leq\smlnorm g_{\Lip}\smlabs{a_{t,k}}\cdot\expect_{t-k}\smlabs{\err_{t-k}-\err^{\ast}}.
\]
Hence, by \eqref{Mnk2} and Jensen's inequality, and recalling that
$\sigma_{\err}^{2}=1$, 
\[
\expect M_{n,k}^{2}\leq2\smlnorm g_{\Lip}^{2}\sum_{t=1}^{n}a_{t,k}^{2},
\]
which together with \eqref{Qndecomp} yields the first inequality
in \eqref{Ebound}.

For the second inequality, we note from \eqref{attk} that
\[
\max_{1\leq t\leq n}\smlabs{a_{t,k}}\leq\sum_{l=0}^{n-1}\smlabs{\rho}^{l}\smlabs{\phi_{k-l}},
\]
with the convention that $\phi_{-l}\defeq0$ for $l<0$. Hence if
$\smlabs{\rho}<1$,
\[
\sum_{k=0}^{\infty}\left(\sum_{t=1}^{n}a_{t,k}^{2}\right)^{1/2}\leq n^{1/2}\sum_{k=0}^{\infty}\max_{1\leq t\leq n}\smlabs{a_{t,k}}\leq n^{1/2}\sum_{l=0}^{n-1}\smlabs{\rho}^{l}\sum_{k=0}^{\infty}\smlabs{\phi_{k-l}}\leq n^{1/2}\frac{\sum_{k=0}^{\infty}\smlabs{\phi_{k}}}{1-\smlabs{\rho}}.\tag*{\qedhere}
\]
\end{proof}

\begin{proof}[Proof of \propref{scaledLLN}\enuref{LLNp1}]
 We take $r=1$ for simplicity; the proof for fixed $r\in[0,1)$
is analogous. When $\rho\in(0,1)$, applying \lemref{Ebound} to the
\emph{unstandardised} process $\{x_{t}\}$ gives the bound
\begin{equation}
\expect\abs{\sum_{t=1}^{n}[g(x_{t})-\expect g(x_{t})]}\leq\smlnorm g_{\Lip}n^{1/2}\frac{\sum_{k=0}^{\infty}\smlabs{\phi_{k}}}{1-\rho}.\label{eq:unstdbnd}
\end{equation}
It follows that replacing $x_{t}$ by the rescaled process $\tilde{x}_{n,t}=\stdseq_{n}^{-1}x_{t}$
in \eqref{unstdbnd} gives
\begin{multline}
\expect\abs{\frac{1}{n}\sum_{t=1}^{n}[g(\tilde{x}_{n,t})-\expect g(\tilde{x}_{n,t})]}\lesssim\frac{1}{n}\cdot\frac{n^{1/2}}{\stdseq_{n}(1-\rho_{n})}\\
\asymp\frac{1}{n^{1/2}}\cdot\frac{(1-\rho_{n}^{2})^{1/2}}{1-\rho_{n}}\asymp\frac{1}{[n(1-\rho_{n})]^{1/2}}=o(1),\label{eq:gdiffnegl}
\end{multline}
where we have used \lemref{variance}.
\end{proof}

\begin{proof}[Proof of \propref{scaledLLN}\enuref{LLNp2}]
 Let $\epsilon>0$. It is proved below that along every sequence
$\{t_{n}\}\subset[n\epsilon,n]$,
\begin{equation}
\tilde{x}_{n,t_{n}}\wkc N[0,1],\label{eq:UCLT}
\end{equation}
whence $\expect g(\tilde{x}_{n,t_{n}})\goesto\int_{\reals}g(x)\gauss(x)\diff x$,
since $g$ is bounded. Then by the preceding and the boundedness of
$g$,
\[
\abs{\frac{1}{n}\sum_{t=1}^{\smlfloor{nr}}\left[\expect g(\tilde{x}_{n,t})-\int g\gauss\right]}\leq\epsilon\smlnorm g_{\infty}+\sup_{t\in[n\epsilon,n]}\abs{\expect g(\tilde{x}_{n,t})-\int g\gauss}\goesto\epsilon\smlnorm g_{\infty}.
\]
Since $\epsilon$ was arbitrary, the result follows.

It remains to prove \eqref{UCLT}. It follows from \lemref{variance}
that $\var(\tilde{x}_{n,t_{n}})\goesto1$. Moreover, we may write
$\tilde{x}_{n,t_{n}}=\sum_{k=-\infty}^{n}\delta_{n,k}\err_{k}$, where
\[
\delta_{n,k}=\begin{cases}
\stdseq_{n}^{-1}a_{t_{n},k} & \text{if }k\leq t_{n},\\
0 & \text{otherwise};
\end{cases}
\]
and
\begin{equation}
\max_{k\leq n}\smlabs{\delta_{n,k}}\leq\stdseq_{n}^{-1}\max_{k\leq t_{n}}\smlabs{a_{t_{n},k}}\leq\stdseq_{n}^{-1}\sum_{i=0}^{\infty}\smlabs{\phi_{i}}=o(1).\label{eq:dnk}
\end{equation}
\eqref{UCLT} therefore follows from Lemma~2.1(i) in \citet*{ADGK14ET}.
\end{proof}

\subsection{Proof of \propref{WP3}}

We shall need the following results, proofs of which appear in Section~S.2
of the Online Supplement. For $\{\rho_{n}\}$ mildly integrated, define
$k_{n}\defeq k_{n}(\{\rho_{n}\})$ to be the largest integer for which
\begin{equation}
k_{n}(\{\rho_{n}\})\leq[(1-\rho_{n})^{-1}\pmin n]/2;\label{eq:kn}
\end{equation}
observe (by \lemref{variance}) that $k_{n}\asymp\stdseq_{n}^{2}$.
Recall the definition of $a_{k}=a_{k}(\rho_{n})$ given immediately
after \eqref{attk} above.
\begin{lem}
\label{lem:coefseq} Suppose $\{\rho_{n}\}$ is mildly integrated.
Then there exist $k_{0},n_{0}\in\naturals$ with $k_{0}$ even, such
that
\begin{enumerate}
\item $\rho_{n}^{k},\rho_{n}^{-k}\in[C_{1},C_{2}]$ for some $C_{1},C_{2}\in(0,\infty)$
for all $n\geq n_{0}$ and $0\leq k\leq2k_{n}$; and
\item for some $\blw a,\abv a\in(0,\infty)$, $\smlabs{a_{0}}\geq\blw a$
and for all $n\geq n_{0}$,
\end{enumerate}
\begin{equation}
\blw a\leq\min_{k_{0}/2\leq k\leq2k_{n}}\smlabs{a_{k}}\leq\max_{0\leq k\leq n}\smlabs{a_{k}}\leq\abv a.\label{eq:coefseqbnd}
\end{equation}
\end{lem}
\begin{lem}
\label{lem:cfbound}Let $\{\vartheta_{k}\}_{k\in\naturals}$ have
$\sigma_{\vartheta}^{2}\defeq\sum_{k=1}^{\infty}\vartheta_{k}^{2}>0$,
and $\{\err_{t}\}_{t\in\integers}$ be as in \enuref{iidseq}. There
exists a bounded function $G(A,\sigma^{2},\psi_{\epsilon})$, not
otherwise depending on $\{\vartheta_{k}\}$, such that $\sigma^{2}\elmap G(A;\sigma^{2},\psi_{\epsilon})$
is decreasing in $\sigma$,
\begin{equation}
\int_{\{\smlabs{\lambda}\geq A\}}\abs{\expect\left(\i\lambda\sum_{k=1}^{\infty}\vartheta_{k}\err_{k}\right)}\diff\lambda\leq G(A;\sigma_{\vartheta}^{2},\psi_{\epsilon})\leq C\sigma_{\vartheta}^{-1},\quad\forall A\geq0\label{eq:cflinprocbnd}
\end{equation}
for some $C<\infty$ depending only on $\smlnorm{\psi_{\epsilon}}_{1}$,
and $\lim_{A\goesto\infty}G(A;\sigma_{\vartheta}^{2},\psi_{\epsilon})=0$.
\end{lem}
\begin{lem}
\label{lem:algebra}Let $\{\rho_{n}\}$ be mildly integrated and $\eta\in(0,1]$.
Then
\[
\frac{1}{n}\int_{1}^{\eta n}\frac{1}{(1-\rho_{n}^{u})^{1/2}}\diff u=\eta+o(1).
\]
\end{lem}
\begin{proof}[Proof of \propref{WP3}]
 We take
\[
d_{n,s,t}\defeq(1-\rho_{n}^{2(t-s)})^{1/2}.
\]
Since $\rho_{n}\goesto1$ with $\rho_{n}<1$, we assume throughout
that $\rho_{n}\in(\frac{1}{2},1)$.

We first consider part\ (a) of \enuref{WP:array}. For (a)(i), we
have
\begin{align*}
\frac{1}{n}\sum_{t=(1-\eta)n}^{n}d_{n,0,t}^{-1} & =\frac{1}{n}\sum_{t=(1-\eta)n}^{n}\frac{1}{(1-\rho_{n}^{2t})^{1/2}}\leq\frac{1}{n}\cdot\frac{\eta n}{(1-\rho_{n}^{2(1-\eta)n})^{1/2}}\goesto\eta
\end{align*}
by \lemref{variance}. For (a)(ii), we note that
\begin{multline*}
\frac{1}{n}\max_{0\leq s\leq(1-\eta)n}\sum_{t=s+1}^{s+\eta n}d_{n,s,t}^{-1}=\frac{1}{n}\sum_{k=1}^{\eta n}\frac{1}{(1-\rho_{n}^{2k})^{1/2}}\leq\frac{1}{n}\sum_{k=1}^{\eta n}\frac{1}{(1-\rho_{n}^{k})^{1/2}}\\
\leq\frac{1}{n}\left\{ \frac{1}{(1-\rho_{n})^{1/2}}+\int_{1}^{\eta n}\frac{1}{(1-\rho_{n}^{u})^{1/2}}\diff u\right\} \goesto\eta,
\end{multline*}
where the final convergence follows by \lemref{algebra}. Finally,
for (a)(iii), essentially the preceding with $\eta=1$ yields
\[
\frac{1}{n}\max_{0\leq s\leq n-1}\sum_{t=s+1}^{n}d_{n,s,t}^{-1}=\frac{1}{n}\sum_{k=1}^{n}\frac{1}{(1-\rho_{n}^{2k})^{1/2}}\goesto1.
\]
Thus part\ (a) of \enuref{WP:array} is satisfied.

We next turn to part~(b) of \enuref{WP:array}. By the Fourier inversion
formula and \lemref{cfbound}, uniform boundedness of $\{h_{n,s,t}\}$
will follow if the variance of $(\tilde{x}_{n,t}-\tilde{x}_{n,s})/d_{n,s,t}$,
conditional on $\tilde{\filt}_{n,s}\defeq\sigma(\{\err_{r}\}_{r\leq s})$,
is bounded away from zero. As in \eqref{xdecmp1} above, we have
\[
x_{t}=\sum_{k=0}^{\infty}a_{t,k}\err_{t-k}=\sum_{k=t-s}^{\infty}a_{t,k}\err_{t-k}+\sum_{k=0}^{t-s-1}a_{k}\err_{t-k}\eqdef x_{s,t}^{\prime}+x_{s+1,t}.
\]
Since $x_{s+1,t,t}$ is independent of $x_{s}$ and $x_{s,t}^{\prime}$,
both of which are $\tilde{\filt}_{n,s}$-measurable, we have
\[
\var\left(\frac{\tilde{x}_{n,t}-\tilde{x}_{n,s}}{d_{n,s,t}}\mid\tilde{\filt}_{n,s}\right)=\var\left(\frac{x_{n,t}-x_{n,s}}{d_{n,s,t}\stdseq_{n}}\mid\tilde{\filt}_{n,s}\right)=\var\left(\frac{x_{s+1,t}}{d_{n,s,t}\stdseq_{n}}\right).
\]
Further, taking $r=t-s$ we have
\[
x_{s+1,t}=\sum_{k=0}^{t-s-1}a_{k}\err_{t-k}\eqdist\sum_{k=0}^{r-1}a_{k}\err_{t-k}=x_{1,r}.
\]
Since $d_{n,s,t}=d_{n,0,t-s}=d_{n,0,r}$, it follows that
\[
\var\left(\frac{x_{s+1,t}}{d_{n,s,t}\stdseq_{n}}\right)=\frac{\var(x_{1,r})}{d_{n,0,r}^{2}\stdseq_{n}^{2}}\geq C\frac{1-\rho_{n}^{2}}{1-\rho_{n}^{2r}}\var(x_{1,r})\eqdef Cg_{n,r}
\]
by \lemref{variance}, for some $C>0$ (depending on $\phi$), for
all $n$ sufficiently large.

We thus need to show that $\inf_{n\geq n_{0}}\inf_{1\leq r\leq n}g_{n,r}>0$
for some $n_{0}\in\naturals$. To that end, we note that for $k_{0}$
as in \lemref{coefseq} and $k_{n}$ as in \eqref{kn},
\begin{equation}
\var(x_{1,r})=\sum_{k=0}^{r}a_{k}^{2}\geq\blw a^{2}\cdot\begin{cases}
1 & \text{if }1\leq r\leq k_{0}\\
r/2 & \text{if }k_{0}+1\leq r\leq k_{n}\\
k_{n}/2 & \text{if }k_{n}+1\leq r\leq n
\end{cases}\label{eq:var1rr}
\end{equation}
for $n$ sufficiently large. We also note the inequality
\[
\frac{1-x^{2}}{1-x^{2r}}=\frac{1}{\sum_{l=0}^{r}x^{2l}}\geq\frac{1}{r},\quad\forall r\in\naturals,x\in(0,1).
\]
Considering each of the three cases in \eqref{var1rr} in turn, we
have:
\begin{enumerate}
\item $1\leq r\leq k_{0}$: then
\[
g_{n,r}\geq\frac{1-\rho_{n}^{2}}{1-\rho_{n}^{2r}}\cdot\blw a^{2}\geq\frac{1}{r}\blw a^{2}\geq\frac{1}{2k_{0}}\blw a^{2};
\]
\item $k_{0}+1\leq r\leq k_{n}$: then
\[
g_{n,r}\geq\frac{1-\rho_{n}^{2}}{1-\rho_{n}^{2r}}\cdot\frac{r}{2}\cdot\blw a^{2}\geq\frac{1}{2}\blw a^{2}\geq\frac{1}{2}\blw a^{2};
\]
\item $k_{n}+1\leq r\leq n$: then for some $C\in(0,\infty)$,
\[
g_{n,r}\geq\frac{1-\rho_{n}^{2}}{1-\rho_{n}^{2r}}\cdot\frac{k_{n}}{2}\cdot\blw a^{2}\geq\frac{C}{(1-\rho_{n}^{2r})}\blw a^{2}\geq\frac{C}{2}\blw a^{2},
\]
where the second inequality follows from $k_{n}\asymp(1-\rho_{n})^{-1}\asymp(1-\rho_{n}^{2})^{-1}$,
and the third inequality from \lemref{coefseq}.
\end{enumerate}
Thus $\inf_{1\leq r\leq n}g_{n,r}$ is bounded away from zero for
$n$ sufficiently large, whence $\{h_{n,s,t}\}$ is uniformly bounded.

Finally, in view of the definition of $\Omega_{n}(\eta)$, \eqref{hequic}
only concerns those $s$ and $t$ for which $(1-\delta)n\geq t-s=r=r_{n}\geq n\delta$
for some $\delta\in(0,1)$. For such $r_{n}$, we have $d_{n,0,r_{n}}=(1-\rho_{n}^{2r_{n}})^{1/2}\goesto1$
by \lemref{variance}, and so arguments given in the proof of \propref{scaledLLN}\enuref{LLNp2}
yield 
\[
z_{n}\defeq\frac{x_{1,r_{n}}}{\stdseq_{n}\cdot d_{n,0,r_{n}}}=(1+o_{p}(1))\cdot\stdseq_{n}^{-1}x_{1,r_{n}}\wkc N[0,1].
\]
Letting $\psi_{z_{n}}$ denote the characteristic function of $z_{n}$,
arguments given in the proof of Corollary~2.2 in \citet{WP09ET}
then imply that \eqref{hequic} holds if the sequence $\{\psi_{z_{n}}\}$
is uniformly integrable. But this is immediate from \lemref{cfbound}
and the fact that $\var(z_{n})\goesto1$, which itself follows from
\lemref{variance}.
\end{proof}
\cleartooddpage{}

\setcounter{page}{1}

\setcounter{section}{19}

\setcounter{subsection}{0}

\setcounter{equation}{0}

\fancyhead[CE]{\scshape supplementary material}

\section*{Online Supplementary Material for `Asymptotic Theory for Kernel
Estimators under Moderate Deviations from a Unit Root' by J.\ A.\ Duffy}

Throughout the following, Assumptions~\assref{reg} and \assref{smoothing}
are always maintained, even when not explicitly referenced. \subsecref{auxunif}
provides the proofs of Lemmas~\ref{lem:std}\textendash \ref{lem:sig2},
and \subsecref{supp:mild} provides the proofs of Lemmas~\ref{lem:variance}
and \ref{lem:coefseq}\textendash \ref{lem:algebra}.

\subsection{Proofs of auxiliary lemmas from \appref{unifproof}\label{subsec:auxunif}}
\begin{proof}[Proof of \lemref{std}]
 Since $\stdseq_{n}^{2}=\var(x_{n})$ is bounded away from zero in
all cases, it suffices to prove that $\stdseq_{n}\lesssim n^{1/2}$
when $\{\rho_{n}\}\in\rseqs$ is mildly integrated or local to unity.
To that end, recall from \eqref{xdecmp1} that
\[
x_{n}=\sum_{k=1}^{n-1}a_{k}\err_{t-k}+\sum_{k=n}^{\infty}a_{t,k}\err_{t-k}
\]
where $a_{t,k}=\sum_{l=0}^{k\pmin(t-1)}\rho^{l}\phi_{k-l}$. Hence
\begin{align*}
\var(x_{n})=\sum_{k=1}^{n-1}a_{t,k}^{2}+\sum_{k=n}^{\infty}a_{t,k}^{2} & \leq\sum_{k=1}^{n-1}\left(\sum_{l=0}^{k}\rho_{n}^{l}\phi_{k-l}\right)^{2}+\sum_{k=n}^{\infty}\left(\sum_{l=0}^{n-1}\rho_{n}^{l}\phi_{k-l}\right)^{2}\\
 & \leq n\left(\sum_{i=0}^{\infty}\smlabs{\phi_{i}}\right)^{2}+\sum_{k=n}^{\infty}\left(\sum_{l=0}^{n-1}\smlabs{\phi_{k-l}}\right)^{2}
\end{align*}
For the second r.h.s. term, we have
\begin{align*}
\sum_{k=n}^{\infty}\left(\sum_{l=0}^{n-1}\smlabs{\phi_{k-l}}\right)^{2}\lesssim\sum_{k=n}^{\infty}\sum_{l=0}^{n-1}\smlabs{\phi_{k-l}} & =\left(\sum_{k=n}^{2n}+\sum_{k=2n+1}^{\infty}\right)\sum_{l=0}^{n-1}\smlabs{\phi_{k-l}}\\
 & \leq\sum_{k=0}^{n}\sum_{l=k}^{\infty}\smlabs{\phi_{l}}+n\sum_{k=n}^{\infty}\smlabs{\phi_{k}}=o(n).\qedhere
\end{align*}
\end{proof}
\begin{proof}[Proof of \lemref{cvg}]
 As noted in the text, the stated convergence follows immediately
from \thmref{fidi}: see also \remref{cvgtozero}. Regarding the strict
positivity of $\tau(x)$: when $\{\rho_{n}\}$ is local to unity,
this follows from \citeauthor{Ray1963IJM}'s \citeyearpar{Ray1963IJM}
theorem; when $\{\rho_{n}\}$ is mildly integrated this is immediate
from $\varphi$ being the standard normal density; and when $\{\rho_{n}\}$
is stationary, this follows from the density $f_{\err}$ of $\err_{t}$
having been assumed strictly positive (see \enuref{iidseq}).
\end{proof}
\begin{proof}[Proof of \lemref{sig2}]
 We first show that $\hat{\regfn}_{n}(x)=\regfn(x)+o_{p}(1)$. To
that end, decompose
\begin{align*}
\hat{\regfn}_{n}(x)-\regfn(x) & =\frac{A_{n,1}+A_{n,2}}{\frac{1}{\scseq_{n}}\sum_{t=1}^{n}K_{h_{n}}(x_{t}-x)}
\end{align*}
where:
\begin{align}
\smlabs{A_{n,1}} & \defeq\frac{1}{\scseq_{n}}\sum_{t=1}^{n}K_{h_{n}}(x_{t}-x)\smlabs{\regfn(x_{t})-\regfn(x)}\nonumber \\
 & \leq\frac{Ch_{n}}{\nseq_{n}}\sum_{t=1}^{n}\frac{1}{h_{n}}K\left(\frac{x_{t}-x}{h_{n}}\right)\abs{\frac{x_{t}-x}{h_{n}}}\nonumber \\
 & \lesssim_{p}h_{n}\label{eq:An1bnd}
\end{align}
by \lemref{cvg} and the Lipschitz continuity of $\regfn$; and
\[
A_{n,2}\defeq\frac{1}{\nseq_{n}}\sum_{t=1}^{n}K_{h_{n}}(x_{t}-x)u_{t+1}=o_{p}(1)
\]
where the claimed negligibility follows since $A_{n,2}$ is a martingale
with variance
\begin{multline*}
\expect A_{n,2}^{2}=\frac{1}{\scseq_{n}^{2}h_{n}^{2}}\sum_{t=1}^{n}\expect K^{2}\left(\frac{x_{t}-x}{h_{n}}\right)u_{t+1}^{2}\\
=\frac{1}{\scseq_{n}h_{n}}\cdot\frac{\sigma^{2}}{\scseq_{n}}\sum_{t=1}^{n}\expect\frac{1}{h_{n}}\sum_{t=1}^{n}K^{2}\left(\frac{x_{t}-x}{h_{n}}\right)\lesssim_{p}\frac{1}{\scseq_{n}h_{n}}=o(1)
\end{multline*}
by \lemref{cvg} and $n^{1/2}h_{n}\goesto\infty$ (see \enuref{smoothing:hn}).
Since by \lemref{cvg}
\[
\frac{1}{\scseq_{n}}\sum_{t=1}^{n}K_{h_{n}}(x_{t}-x)\wkc\tau(x)
\]
which is a.s.\ positive, we have $\hat{\regfn}_{n}(x)=\regfn(x)+o_{p}(1)$
as claimed.

The remainder of the proof follows similar lines to the proof of Theorem~3.2
in \citet{WP09Ecta}. Recalling
\[
\hat{\sigma}_{u}^{2}(x)=\frac{\frac{1}{\scseq_{n}}\sum_{t=1}^{n}K_{h_{n}}(x_{t}-x)[y_{t+1}-\hat{\regfn}_{n}(x)]^{2}}{\frac{1}{\scseq_{n}}\sum_{t=1}^{n}K_{h_{n}}(x_{t}-x)}
\]
we decompose the numerator as
\begin{align*}
 & \frac{1}{\scseq_{n}}\sum_{t=1}^{n}K_{h_{n}}(x_{t}-x)[y_{t+1}-\hat{\regfn}_{n}(x)]^{2}\\
 & \qquad\qquad=\frac{1}{\scseq_{n}}\sum_{t=1}^{n}K_{h_{n}}(x_{t}-x)u_{t+1}^{2}+\frac{2}{\scseq_{n}}\sum_{t=1}^{n}K_{h_{n}}(x_{t}-x)[\regfn(x_{t})-\hat{\regfn}_{n}(x)]u_{t+1}\\
 & \qquad\qquad\qquad\qquad+\frac{1}{\scseq_{n}}\sum_{t=1}^{n}K_{h_{n}}(x_{t}-x)[\regfn(x_{t})-\hat{\regfn}_{n}(x)]^{2}\\
 & \qquad\qquad\eqdef B_{n,1}+2B_{n,2}+B_{n,3}.
\end{align*}

Letting $\zeta_{t}\defeq u_{t}^{2}-\sigma_{u}^{2}$, we claim that
\begin{align}
B_{n,1} & =\frac{\sigma_{u}^{2}}{\scseq_{n}}\sum_{t=1}^{n}K_{h_{n}}(x_{t}-x)+\frac{1}{\scseq_{n}}\sum_{t=1}^{n}K_{h_{n}}(x_{t}-x)\zeta_{t+1}\label{eq:An1}\\
 & \wkc\sigma_{u}^{2}\tau(x).\nonumber 
\end{align}
The convergence of the first r.h.s.\ term in \eqref{An1} follows
from \lemref{cvg}. Regarding the second r.h.s.\ term, we note that
since $\zeta_{t+1}\defeq u_{t+1}^{2}-\sigma_{u}^{2}$ is a martingale
difference under \enuref{regdist}, this term is a martingale with
conditional variance
\[
\frac{1}{e_{n}h_{n}}\cdot\frac{1}{e_{n}}\sum_{t=1}^{n}\frac{1}{h_{n}}K^{2}\left(\frac{x_{t}-x}{h_{n}}\right)\expect[\zeta_{t+1}^{2}\mid\filtg_{t}]\lesssim_{p}\frac{1}{e_{n}h_{n}}=o(1)
\]
by \lemref{cvg} and $\sup_{t}\expect[\zeta_{t+1}^{2}\mid\filtg_{t}]<\infty$
a.s.\ (under \enuref{regdist}). It follows by Corollary 3.1 of \citet{HH80}
that, indeed,
\[
\frac{1}{\scseq_{n}}\sum_{t=1}^{n}K_{h_{n}}(x_{t}-x)\zeta_{t+1}\inprob0.
\]

Next, we have
\begin{align*}
B_{n,3} & \leq C\frac{1}{\scseq_{n}}\sum_{t=1}^{n}K_{h_{n}}(x_{t}-x)\left\{ [\regfn(x_{t})-\regfn(x)]^{2}+[\hat{\regfn}_{n}(x)-\regfn(x)]^{2}\right\} \\
 & =O_{p}(h_{n}^{2})+o_{p}(1)\\
 & =o_{p}(1)
\end{align*}
by an analogous argument as was used to prove \eqref{An1bnd}, and
$\hat{\regfn}_{n}(x)=\regfn(x)+o_{p}(1)$. Finally
\[
B_{n,2}\leq(B_{n,1})^{1/2}(B_{n,3})^{1/2},
\]
by the Cauchy-Schwarz inequality; whence by \lemref{cvg} and the
preceding,
\[
\hat{\sigma}_{u}^{2}(x)=\frac{B_{n,1}+B_{n,2}+B_{n,3}}{\frac{1}{\scseq_{n}}\sum_{t=1}^{n}K_{h_{n}}(x_{t}-x)}\wkc\frac{\sigma_{u}^{2}\tau(x)}{\tau(x)}=\sigma_{u}^{2}.\tag*{\qedhere}
\]
\end{proof}

\subsection{Proofs of auxiliary lemmas from \appref{mild}\label{subsec:supp:mild}}
\begin{proof}[Proof of \lemref{variance}]
~ Letting $c_{n}\defeq n(\rho_{n}-1)\goesto-\infty$, we note that
for every $M<\infty$, we may take $n$ sufficiently large such that
$c_{n}<-M$, whence 
\[
\rho_{n}^{n\epsilon}=\left(1+\frac{c_{n}}{n}\right)^{n\epsilon}\leq\left(1-\frac{M}{n}\right)^{n\epsilon}\goesto\e^{-M\epsilon}\goesto0
\]
as $n\goesto\infty$ and then $M\goesto\infty$. Thus (i) holds.

Now taking $s=1$ in \eqref{xdecmp1}, we have
\[
x_{t}=\sum_{k=0}^{t-1}a_{k}\err_{t-k}+\sum_{k=t}^{\infty}a_{t,k}\err_{t-k}=x_{1,t}+x_{0,t}^{\prime}
\]
where $x_{1,t}$ and $x_{0,t}^{\prime}$ are independent, with variances
$\varsigma_{1,t}^{2}\defeq\var(x_{1,t})$ and $\varsigma_{2,t}^{2}\defeq\var(x_{0,t}^{\prime})$
respectively. Let $\{t_{n}\}\subseteq[n\epsilon,n]$ be as in the
statement of part~(iii) of the lemma. We shall prove below that
\begin{equation}
(1-\rho_{n}^{2})\var(x_{t_{n}})=(1-\rho_{n}^{2})(\varsigma_{1,t_{n}}^{2}+\varsigma_{2,t_{n}}^{2})=(1-\rho_{n}^{2})\varsigma_{1,t_{n}}^{2}+o(1)\goesto\phi^{2},\label{eq:varlimit}
\end{equation}
from which both parts (ii) and (iii) of the lemma immediately follow.

Some tedious algebra (verified immediately below this proof) yields
\begin{align}
\varsigma_{1,t_{n}}^{2} & =\sum_{k=0}^{t_{n}-1}\left(\sum_{l=0}^{k}\rho_{n}^{k-l}\phi_{l}\right)^{2}=\sum_{i=0}^{t_{n}-1}\phi_{i}^{2}\sum_{k=0}^{t_{n}-i-1}\rho_{n}^{2k}+2\sum_{i=0}^{t_{n}-1}\sum_{j=i+1}^{t_{n}-1}\phi_{i}\phi_{j}\sum_{k=0}^{t_{n}-j-1}\rho_{n}^{2k+(j-i)}\ \ \label{eq:tedious}
\end{align}
whence, since $\rho_{n}\in(0,1)$,
\begin{align*}
(1-\rho_{n}^{2})\varsigma_{1,t_{n}}^{2} & =\sum_{i=0}^{t_{n}-1}\phi_{i}^{2}(1-\rho_{n}^{2(t_{n}-i)})+2\sum_{i=0}^{t_{n}-1}\sum_{j=i+1}^{t_{n}-1}\phi_{i}\phi_{j}(1-\rho_{n}^{2(t_{n}-j)+(j-i)})
\end{align*}
Since $\rho_{n}^{2(t_{n}-i)}\leq\rho_{n}^{2(\smlfloor{n\epsilon}-i)}\goesto0$
as $n\goesto\infty$ for each \emph{fixed} $i\in\naturals$ by part~(i),
and $\sum_{i=0}^{\infty}\smlabs{\phi_{i}}<\infty$, it follows that
\[
(1-\rho_{n}^{2})\varsigma_{1,t_{n}}^{2}\goesto\sum_{i=0}^{\infty}\phi_{i}^{2}+2\sum_{i=0}^{\infty}\sum_{j=i+1}^{\infty}\phi_{i}\phi_{j}=\phi^{2}.
\]

Regarding $\varsigma_{2,t_{n}}^{2}$, we note that since $\smlabs{\rho_{n}}\leq1$
and $C_{\phi}\defeq\sum_{i=0}^{\infty}\smlabs{\phi_{i}}<\infty$
\begin{align*}
\varsigma_{2,t_{n}}^{2} & =\sum_{k=t_{n}}^{\infty}\left(\sum_{l=0}^{t_{n}-1}\rho^{l}\phi_{k-l}\right)^{2}\leq C_{\phi}\sum_{k=t_{n}}^{\infty}\sum_{l=0}^{t_{n}-1}\rho_{n}^{l}\smlabs{\phi_{k-l}}\leq C_{\phi}\sum_{l=0}^{t_{n}-1}\rho_{n}^{l}\tilde{\phi}_{t_{n}-l},
\end{align*}
where $\tilde{\phi}_{j}\defeq\sum_{i=j}^{\infty}\smlabs{\phi_{i}}$.
Further,
\begin{multline*}
\sum_{l=0}^{t_{n}-1}\rho_{n}^{l}\tilde{\phi}_{t_{n}-l}=\left(\sum_{l=0}^{\smlfloor{t_{n}/2}-1}+\sum_{l=\smlfloor{t_{n}/2}}^{t_{n}-1}\right)\rho_{n}^{l}\tilde{\phi}_{t_{n}-l}\\
\leq\left(\tilde{\phi}_{\smlfloor{t_{n}/2}-1}+C_{\phi}\rho_{n}^{\smlfloor{t_{n}/2}}\right)\sum_{l=0}^{\smlfloor{t_{n}/2}-1}\rho_{n}^{l}=o[(1-\rho_{n}^{2})^{-1}],
\end{multline*}
since $\tilde{\phi}_{\smlfloor{t_{n}/2}}\goesto0$ and $\rho_{n}^{\smlfloor{t_{n}/2}}\goesto0$
by part~(i), and
\[
\sum_{l=0}^{\smlfloor{t_{n}/2}}\rho_{n}^{l}\leq(1-\rho_{n})^{-1}\asymp(1-\rho_{n}^{2})^{-1},
\]
whence $\varsigma_{2,t_{n}}^{2}=o[(1-\rho_{n}^{2})^{-1}]$.
\end{proof}
\begin{proof}[Verification of \eqref{tedious}]
 Dropping the $n$ subscript from $t_{n}$ and $\rho_{n}$ for simplicity,
and setting $m\defeq t-1$, we have

\begin{align*}
\sum_{k=0}^{m}\left(\sum_{l=0}^{k}\rho^{k-l}\phi_{l}\right)^{2} & =\sum_{k=0}^{m}\sum_{i=0}^{k}\sum_{j=0}^{k}\rho^{2k-i-j}\phi_{i}\phi_{j}\\
 & =\sum_{i=0}^{m}\sum_{j=0}^{m}\phi_{i}\phi_{j}\sum_{k=i\pmax j}^{m}\rho^{2k-i-j}\\
 & =\sum_{i=0}^{m}\phi_{i}^{2}\sum_{k=i}^{m}\rho^{2(k-i)}+2\sum_{i=0}^{m}\sum_{j=i+1}^{m}\phi_{i}\phi_{j}\sum_{k=j}^{m}\rho^{2(k-j)+(j-i)}\\
 & =\sum_{i=0}^{m}\phi_{i}^{2}\sum_{k=0}^{m-i}\rho^{2k}+2\sum_{i=0}^{m}\sum_{j=i+1}^{m}\phi_{i}\phi_{j}\sum_{k=0}^{m-j}\rho^{2k+(j-i)}.\qedhere
\end{align*}
\end{proof}
\begin{proof}[Proof of \lemref{coefseq}]
 When $\{\rho_{n}\}$ is mildly integrated, $\rho_{n}\in(0,1)$ and
the upper bound in \eqref{coefseqbnd} follows trivially from $\smlabs{a_{k}(\rho_{n})}\leq\sum_{i=0}^{\infty}\smlabs{\phi_{i}}$.
Further, for any $0\leq k\leq2k_{n}$, 
\[
\rho_{n}^{2k_{n}}\leq\rho_{n}^{k}\leq\rho_{n}^{-k}\leq\rho_{n}^{-2k_{n}}.
\]
Noting that $\rho^{(1-\rho)^{-1}}\goesto\e^{-1}$ as $\rho\goesto1$,
and $2k_{n}\sim(1-\rho_{n})^{-1}$, it follows that $(\rho_{n}^{2k_{n}},\rho_{n}^{-2k_{n}})\goesto(\e^{-1},\e)$
as $n\goesto\infty$. Thus there exists an $n_{0}\in\naturals$ and
$C_{1},C_{2}\in(0,\infty)$ such that $\rho_{n}^{k},\rho_{n}^{-k}\in[C_{1},C_{2}]$
for all $n\geq n_{0}$ and $0\leq k\leq2k_{n}$.

Now $a_{k}(\rho_{n})=\rho_{n}^{k}\sum_{l=0}^{k}\rho_{n}^{-l}\phi_{l}$,
and for any $m\leq k\leq2k_{n}$,
\[
\sum_{l=0}^{k}\rho_{n}^{-l}\phi_{l}=\sum_{l=0}^{m}\phi_{l}-\sum_{l=0}^{m}(1-\rho_{n}^{-l})\phi_{l}+\sum_{l=m+1}^{k}\rho_{n}^{-l}\phi_{l}.
\]
Therefore, since $\smlabs{\rho_{n}^{k}}\leq1$, 
\begin{align*}
\abs{a_{k}(\rho_{n})-\rho_{n}^{k}\sum_{l=0}^{m}\phi_{l}} & \leq\sum_{l=0}^{m}\smlabs{1-\rho_{n}^{-l}}\smlabs{\phi_{l}}+\sum_{l=m+1}^{k}\smlabs{\phi_{l}}
\end{align*}
Let $m_{0}$ be chosen such that both
\[
\rho_{n}^{k}\abs{\sum_{l=0}^{m_{0}}\phi_{l}}\geq C_{1}\abs{\sum_{l=0}^{m_{0}}\phi_{l}}\geq\frac{C_{1}}{2}\smlabs{\phi}\eqdef3\blw a
\]
for all $n\geq n_{0}$, and $\sum_{l=m_{0}+1}^{\infty}\smlabs{\phi_{l}}\leq\blw a$.
Since $\rho_{n}^{-l}\goesto1$ for each $l$, there exists an $n_{1}\ge n_{0}$
such that
\begin{align*}
\smlabs{a_{k}(\rho_{n})} & \geq\rho_{n}^{k}\abs{\sum_{l=0}^{m_{0}}\phi_{l}}-\sum_{l=0}^{m_{0}}\smlabs{1-\rho_{n}^{-l}}\smlabs{\phi_{l}}-\sum_{l=m_{0}+1}^{k}\smlabs{\phi_{l}}\geq\blw a
\end{align*}
for all $n\geq n_{1}$. Taking $k_{0}\defeq2m_{0}$ and re-designating
$n_{1}$ as $n_{0}$ gives the claimed lower bound in \eqref{coefseqbnd}.

Finally, since $a_{0}=\phi_{0}$ is nonzero by \enuref{regproc},
replacing $\blw a$ by $\blw a\pmin\smlabs{\phi_{0}}$ yields a lower
bound that also applies to $\smlabs{a_{0}}$.
\end{proof}
\begin{proof}[Proof of \lemref{cfbound}]
 Since $\psi_{\err}\in L^{1}$, $\err_{0}$ has a bounded continuous
density. Thus by the Riemann-Lebesgue lemma (\citealp{Feller2}, Lem.~XV.3.3)
$\limsup_{\smlabs{\lambda}\goesto\infty}\smlabs{\psi_{\err}(\lambda)}=0$.
Further, $\psi_{\err}\in L^{1}$ cannot be periodic, and so $\smlabs{\psi_{\err}(\lambda)}<1$
for all $\lambda\neq0$ (\citealp{Feller2}, Lem.~XV.1.4); since
$\psi_{\err}$ is necessarily continuous \citep[Lem.~XV.1.1]{Feller2},
it follows that $\sup_{\smlabs{\lambda}\geq1}\smlabs{\psi_{\epsilon}(\lambda)}\geq\e^{-\gamma_{0}}$
for some $\gamma_{0}\in(0,\infty)$. By the moments theorem for characteristic
functions \citep[Lem.~XV.4.2]{Feller2}, we have $\psi_{\err}(\lambda)=1-\frac{1}{2}\lambda^{2}(1+o(1))$
as $\lambda\goesto0$. Thus there exists a $\gamma_{1}\in(0,\infty)$
such that $\smlabs{\psi_{\epsilon}(\lambda)}\leq\e^{-\gamma_{1}\lambda^{2}}$.
Taking $\gamma\defeq\gamma_{0}\pmin\gamma_{1}$ thus gives
\begin{equation}
\smlabs{\psi_{\epsilon}(\lambda)}\leq\begin{cases}
\e^{-\gamma\lambda^{2}} & \text{if }\smlabs{\lambda}\in[0,1],\\
\e^{-\gamma} & \text{if }\smlabs{\lambda}\geq1.
\end{cases}\label{eq:cferrbnd}
\end{equation}

Let $\psi_{\vartheta}(\lambda)\defeq\expect\exp(\i\lambda\sum_{k=1}^{\infty}\vartheta_{k}\err_{k})=\prod_{k=1}^{\infty}\psi_{\epsilon}(\vartheta_{k}\lambda)$;
we want to control the integral of (the modulus of) this function
over $[A,\infty)$. Without loss of generality, assume the coefficients
$\{\vartheta_{k}\}$ are ordered such that $\smlabs{\vartheta_{i}}\geq\smlabs{\vartheta_{i+1}}$.
Since 
\[
\sum_{k=1}^{\infty}\frac{3\sigma_{\vartheta}^{2}}{\pi}\cdot k^{-2}=\frac{\sigma_{\vartheta}^{2}}{2}=\frac{1}{2}\sum_{k=1}^{\infty}\vartheta_{k}^{2},
\]
the set
\[
\mathcal{K}\defeq\left\{ k\in\naturals\mid\vartheta_{k}^{2}\geq\frac{3\sigma_{\vartheta}^{2}}{\pi}\cdot k^{-2}\right\} 
\]
must be nonempty; let $k^{\ast}$ denote the smallest element of $\mathcal{K}$.

We will bound the integral of $\smlabs{\psi_{\vartheta}}$ separately
over each of the two r.h.s.\ sets in
\begin{equation}
[A,\infty)=[A,A\pmax\vartheta_{k^{\ast}}^{-1}]\union[A\pmax\vartheta_{k^{\ast}}^{-1},\infty).\label{eq:Asplit}
\end{equation}
We first have
\begin{align*}
\int_{\{\smlabs{\lambda}\in[A,A\pmax\vartheta_{k^{\ast}}^{-1}]\}}\smlabs{\psi_{\vartheta}(\lambda)}\diff\lambda & \leq\int_{\{\smlabs{\lambda}\in[A,A\pmax\vartheta_{k^{\ast}}^{-1}]\}}\prod_{k\in\mathcal{K}}\smlabs{\psi_{\epsilon}(\vartheta_{k}\lambda)}\diff\lambda\\
 & \leq_{(2)}\int_{\{\smlabs{\lambda}\geq A\}}\exp\left(-\gamma\lambda^{2}\sum_{k\in\mathcal{K}}\vartheta_{k}^{2}\right)\diff\lambda\\
 & \leq_{(3)}\int_{\{\smlabs{\lambda}\geq A\}}\exp(-\gamma\lambda^{2}\sigma_{\vartheta}^{2}/2)\diff\lambda
\end{align*}
where $\leq_{(2)}$ follows from \eqref{cferrbnd} and
\[
\smlabs{\lambda}\in[A,A\pmax\vartheta_{k^{\ast}}^{-1}]\implies\smlabs{\vartheta_{k^{\ast}}\lambda}\leq1\implies\smlabs{\vartheta_{k}\lambda}\leq1,\quad\forall k\geq k^{\ast};
\]
while $\leq_{(3)}$ follows from
\[
\sum_{k\in\mathcal{K}}\vartheta_{k}^{2}=\sigma_{\vartheta}^{2}-\sum_{k\notin\mathcal{K}}\vartheta^{2}\geq\sigma_{\vartheta}^{2}-\frac{3\sigma_{\vartheta}^{2}}{\pi}\cdot\sum_{k\notin\mathcal{K}}k^{-2}\geq\frac{\sigma_{\vartheta}^{2}}{2}.
\]

Next, we have
\begin{align*}
\int_{\{\smlabs{\lambda}\in[A\pmax\vartheta_{k^{\ast}}^{-1},\infty)\}}\smlabs{\psi_{\vartheta}(\lambda)}\diff\lambda & \leq\int_{\{\smlabs{\lambda}\in[A\pmax\vartheta_{k^{\ast}}^{-1},\infty)\}}\prod_{k=1}^{k^{\ast}}\psi_{\epsilon}(\vartheta_{k}\lambda)\diff\lambda\\
 & \leq_{(2)}\e^{-\gamma(k^{\ast}-1)}\int_{\{\smlabs{\lambda}\in[A\pmax\vartheta_{k^{\ast}}^{-1},\infty)\}}\smlabs{\psi_{\epsilon}(\vartheta_{k^{\ast}}\lambda)}\diff\lambda\\
 & \leq\e^{-\gamma(k^{\ast}-1)}\int_{\{\smlabs{\lambda}\geq A\}}\smlabs{\psi_{\epsilon}(\vartheta_{k^{\ast}}\lambda)}\diff\lambda\\
 & =\e^{-\gamma(k^{\ast}-1)}\vartheta_{k^{\ast}}^{-1}\int_{\{\smlabs{\lambda}\geq\vartheta_{k^{\ast}}A\}}\smlabs{\psi_{\epsilon}(\lambda)}\diff\lambda\\
 & \leq_{(5)}c_{0}^{-1}\sigma_{\vartheta}^{-1}\e^{-\gamma(k^{\ast}-1)}k^{\ast}\int_{\{\smlabs{\lambda}\geq c_{0}\sigma_{\vartheta}A/k^{\ast}\}}\smlabs{\psi_{\epsilon}(\lambda)}\diff\lambda,
\end{align*}
for $c_{0}\defeq(3/\pi)^{1/2}$, where $\leq_{(2)}$ holds trivially
if $k^{\ast}=1$, and otherwise follows from
\[
\smlabs{\lambda}\in[A\pmax\vartheta_{k^{\ast}}^{-1},\infty)\implies\smlabs{\vartheta_{k^{\ast}}\lambda}\geq1\implies\smlabs{\vartheta_{k}\lambda}\geq1,\quad\forall k\leq k^{\ast};
\]
while $\leq_{(5)}$ follows from $\vartheta_{k^{\ast}}^{2}\geq(3\sigma_{\vartheta}^{2}/\pi)\cdot(k^{\ast})^{-2}$.

Finally, define
\begin{multline*}
G(A;\sigma^{2},\psi_{\epsilon})\\
\defeq\int_{\{\smlabs{\lambda}\geq A\}}\exp(-\gamma\lambda^{2}\sigma^{2}/2)\diff\lambda+c_{0}^{-1}\sigma^{-1}\sup_{k\geq1}\e^{-\gamma(k-1)}k\int_{\{\smlabs{\lambda}\geq c_{0}\sigma A/k\}}\smlabs{\psi_{\epsilon}(\lambda)}\diff\lambda,
\end{multline*}
which clearly satisfies the first inequality in \eqref{cflinprocbnd},
and is decreasing in $\sigma^{2}$; the second inequality in \eqref{cflinprocbnd}
follows by evaluating $G(0;\sigma^{2},\psi_{\epsilon})$, and noting
$\sup_{k\geq1}\e^{-\gamma(k-1)}k<\infty$. It thus remains to show
that $G(A;\sigma^{2},\psi_{\epsilon})\goesto0$ as $A\goesto\infty$.
To that end, let $\epsilon>0$ and note that there exists a $k^{\prime}$
such that
\[
\e^{-\gamma(k^{\prime}-1)}k^{\prime}\int_{\reals}\smlabs{\psi_{\epsilon}(\lambda)}\diff\lambda<\epsilon.
\]
Since
\[
\e^{-\gamma(k-1)}k\int_{\{\smlabs{\lambda}\geq c_{0}\sigma A/k\}}\smlabs{\psi_{\epsilon}(\lambda)}\diff\lambda\goesto0
\]
as $A\goesto\infty$, for each fixed $k\in\{1,\ldots,k^{\prime}\}$,
the claim follows. 
\end{proof}
\begin{proof}[Proof of \lemref{algebra}]
 Making the change of variables $u=\rho^{x}$, we have
\[
\int_{1}^{a}\frac{1}{(1-\rho^{x})^{1/2}}\diff x=\frac{1}{-\log\rho}\int_{\rho^{a}}^{\rho}\frac{1}{(1-u)^{1/2}u}\diff u=\frac{1}{-\log\rho}\left[-2\tanh^{-1}\{(1-u)^{1/2}\}\right]_{\rho^{a}}^{\rho}.
\]
for $\rho\in(0,1)$, where $\tanh^{-1}(x)\defeq\frac{1}{2}\log\{(1+x)/(1-x)\}$
is inverse hyperbolic tangent function. Now set $\rho=\rho_{n}$,
for $\{\rho_{n}\}$ mildly integrated, and $a=n\eta$: and note that
$\rho_{n}\goesto1$, whereas $\rho_{n}^{\eta n}\goesto0$ by \lemref{variance}.
Then
\begin{align*}
\frac{1}{n}\int_{1}^{\eta n}\frac{1}{(1-\rho_{n}^{x})^{1/2}}\diff x & =\frac{1}{n}\cdot\frac{1}{-\log\rho_{n}}\left\{ 2\tanh^{-1}[(1-\rho_{n}^{\eta n})^{1/2}]+o(1)\right\} \\
 & \sim\frac{1}{n}\cdot\frac{\log[1-(1-\rho_{n}^{\eta n})^{1/2}]}{\log\rho_{n}}.
\end{align*}
Next, note that by two applications of L'Hôpital's rule
\begin{multline*}
\lim_{x\goesto0}\frac{\log[1-(1-x)^{1/2}]}{\log x}=\lim_{x\goesto0}\frac{\frac{1}{2}(1-x)^{-1/2}/[1-(1-x)^{1/2}]}{1/x}\\
=\frac{1}{2}\lim_{x\goesto0}\frac{x}{1-(1-x)^{1/2}}=\frac{1}{2}\lim_{x\goesto0}\frac{1}{\frac{1}{2}(1-x)^{-1/2}}=1,
\end{multline*}
whence
\[
\frac{1}{n}\cdot\frac{\log[1-(1-\rho_{n}^{\eta n})^{1/2}]}{\log\rho_{n}}\sim\frac{1}{n}\cdot\frac{\log(\rho_{n}^{\eta n})}{\log\rho_{n}}=\eta
\]
and the result follows.
\end{proof}

\end{document}